\documentclass[12pt, reqno]{amsart}
\usepackage{mathtools, amsthm, amscd, amsfonts, amssymb}
\usepackage{bm}
\usepackage{graphicx}
\usepackage[shortlabels]{enumitem}
\usepackage{mathrsfs}
\usepackage{xcolor}
\usepackage{cite}

\usepackage{lineno}
\usepackage[
colorlinks=true,
linkcolor=blue,
citecolor=blue,
urlcolor=blue,
pagebackref=true
]{hyperref}

\newtheorem{theorem}{Theorem}[section]
\newtheorem{lemma}[theorem]{Lemma}
\newtheorem{proposition}[theorem]{Proposition}
\newtheorem{corollary}[theorem]{Corollary}
\theoremstyle{definition}
\newtheorem{definition}[theorem]{Definition}

\theoremstyle{remark}

\numberwithin{equation}{section}

\numberwithin{equation}{section}

\newcommand{\mc}{\mathcal}
\newcommand{\mathcalH}{\mathcal{H}}

\DeclareMathOperator{\im}{range}

\newcommand{\gt}{\Gamma\otimes}

\newcommand{\ov}{\overline}
\newcommand{\sub}{\subseteq}

\newcommand{\wh}{\widehat}
\newcommand{\wt}{\widetilde}

\newcommand{\ttt}{T = [T_1, \cdots, T_n]}
\newcommand{\vvv}{V = [V_1, \cdots, V_n]}

\newcommand{\SA}{\sum_{\alpha \in \mathbb{F}_k^+}}

\DeclareMathOperator{\diag}{diag}

\begin{document}
	
	\title[$k$-Regular Factorizations and Joint Invariant Subspaces ]{$k$-Regular Factorizations and Joint Invariant Subspaces of Completely Non-Coisometric Row Contractions}
	
	%	Author information
	\author[Kalpesh J. Haria]{Kalpesh J. Haria}
	\address{School of Mathematics and Computer Science, Indian Institute of Technology Goa, Goa 403401, India}
	\email{kalpesh@iitgoa.ac.in, hikalpesh.haria@gmail.com}
	
	\author[Aashish Kumar Maurya]{Aashish Kumar Maurya}
	\address{School of Mathematics and Computer Science, Indian Institute of Technology Goa, Goa 403401, India}
	\email{aashish21232101@iitgoa.ac.in, a.k.maurya.math@gmail.com}

	% Subject classification and keywords - MODIFY THIS SECTION

	\begin{abstract}
		This article investigates $k$-regular factorizations of characteristic functions associated with completely non-coisometric row contractions. In this setting, a one-to-one correspondence is established between chains of joint invariant subspaces
		\[
		\mathcal{M}_1 \sub \cdots \sub \mathcal{M}_{k-1}
		\]
		and $k$-regular factorizations of the characteristic function of a completely non-coisometric row contraction. A functional model corresponding to a given $k$-regular factorization of a purely contractive multi-analytic operator satisfying the Szeg\H{o} condition is further constructed, and the associated chain of joint invariant subspaces is characterized in terms of the underlying multi-analytic factors. Finally, it is shown that any such chain of joint invariant subspaces induces a block upper-triangular decomposition of the underlying row contraction,  and that the characteristic function of each diagonal block coincides with the purely contractive part of the corresponding factor in the $k$-regular factorization.
	\end{abstract}
	
	\maketitle
	
	\tableofcontents
	\section{Introduction}
	
	%%%%%%%%%%%%%%%%%%%%%%%%%%%%%%%%%%%%%%%%%%%%%%%%%%%%%%
	Throughout this article, all Hilbert spaces are assumed to be separable complex Hilbert spaces.
	A row contraction on a Hilbert space \(\mathcal{H}\) is an \(n\)-tuple of operators 
	\(T = [T_1, \dots, T_n]: \bigoplus_{i=1}^n\mathcal{H} \to \mathcal{H}\) satisfying the inequality 
	\(\sum_{i=1}^n T_i T_i^* \leq I_{\mathcal{H}}\). A row contraction \(T\) is said to be 
	\emph{completely non-coisometric} (c.n.c.) if the subspace \(\mathcal{H}_c\), defined by
	\[
	\mathcal{H}_c \coloneqq \left\{ h \in \mathcal{H} : \sum_{|\alpha| = k} \|T_{\alpha}^* h\|^2 = \|h\|^2 \quad \text{for all } k \geq 1 \right\},
	\]
	reduces to the trivial subspace \(\{0\}\). A closed subspace $\mathcal M \sub \mathcal H$ is said to be 
	\emph{joint invariant} with respect to $\ttt$ if it is invariant under each $T_i$ for 
	$i=1,\dots,n.$
	
	For an integer $k \geq 2$, the concept of a $k$-regular factorization for the product of $k$ contractions was introduced in \cite{HM2026I}. This notion generalizes the classical regular factorization for the product of two contractions established by Sz.-Nagy and Foia\c{s} \cite{Sz64b,Sz64a,NF74}. By applying the theory of $k$-regular factorizations to the characteristic functions of completely non-unitary (c.n.u.) contractions, a one-to-one correspondence was established in \cite{HM2026I} between chains of invariant subspaces of a c.n.u.\ contraction $T$ and the $k$-regular factorizations of its characteristic function $\Theta_T$. Furthermore, a functional model for such contractions was constructed, together with functional model representations for the associated chains of invariant subspaces.
	
	The general theory of characteristic functions and functional models for row contractions was developed in the foundational works of A.~E.~Frazho \cite{Fr82a} and G.~Popescu \cite{Po95a,Po89b,Po89a,Po06}. For the corresponding theory in the setting of commutative row contractions, we refer to the works of T.~Bhattacharyya, J.~Eschmeier, and J.~Sarkar \cite{Bh05a,Bh06a}. In this multivariable setting, Popescu's characteristic function of a row contraction is a purely contractive multi-analytic operator. Notably, in \cite{Po06}, G.~Popescu established a one-to-one correspondence between joint invariant subspaces and the regular factorizations of the characteristic function of a c.n.c.\ row contraction, extending the classical theory of Sz.-Nagy and Foia\c{s} for c.n.u.\ contractions given in  \cite{Sz64b,Sz64a,NFBK10} to the multivariable setting.
	
	In the present article, the results of \cite{HM2026I} are extended to the multivariable setting by applying the concept of $k$-regular factorizations to contractive multi-analytic operators. Specifically,  a one-to-one correspondence is established between chains of joint invariant subspaces 
	\[
	\mathcal{M}_1 \sub \cdots \sub \mathcal{M}_{k-1}
	\]
	and $k$-regular factorizations of the characteristic function (i.e., $\Theta_T = \Theta_k \cdots \Theta_1$) of a c.n.c.\ row contraction. Furthermore, a functional model associated with a given $k$-regular factorization of a purely contractive multi-analytic operator satisfying the Szeg\H{o} condition (see \cite[Condition 3.2]{Po06}) is constructed, and obtain corresponding functional model representations for the associated chains of joint invariant subspaces of the c.n.c.\ row contraction. Finally, for a given $k$-regular factorization, the associated chain of joint invariant subspaces induces a natural upper triangular block decomposition of the c.n.c.\ row contraction. More importantly, we show that the characteristic function of each diagonal block coincides with the purely contractive part of the corresponding factor $\Theta_i$.	
	
	A contraction $A_1 \in B(\mathcal H,\mathcal F)$ is called a \emph{divisor} of a contraction $A \in B(\mathcal H,\mathcal K)$ if there exists a contraction $A_2 \in B(\mathcal F,\mathcal K)$ such that $A = A_2 A_1$, and if  the factorization $A = A_2 A_1$ is regular, then $A_1$ is called  \emph{regular divisor}. In Proposition 2.4 of Chapter VII in \cite{NF70}, Sz.-Nagy and Foia\c{s} proved that if \( \mathcal{M} \) and \( \mathcal{M}' \) are invariant subspaces of a c.n.u. contraction \( T \), corresponding to regular factorizations
	\(\Theta_T(z) = \Theta_2(z)\Theta_1(z)
	\text{ and } 
	\Theta_T(z) = \Theta_2'(z)\Theta_1'(z),\)
	respectively, then the inclusion $\mathcal M \sub \mathcal M'$ implies that $\Theta_{1}$ is a divisor of $\Theta'_1$. In a subsequent article \cite{Kerchy03}, L. K\'erchy refined this result by demonstrating that $\Theta_{1}$ is, in fact, a regular divisor of $\Theta'_1$.
	
	In the multivariable setting, it was shown in \cite{Po06} that if $\mathcal M$ and $\mathcal M'$ are joint invariant subspaces of a c.n.c. row contraction $T$, corresponding to the regular factorizations $\Theta_T=\Theta_2\Theta_1$ and $\Theta_T=\Theta_2'\Theta_1'$ respectively, then the inclusion $\mathcal M\sub \mathcal M'$ implies that $\Theta_{1}$ is a divisor of $\Theta'_1$. In Section 4 of the present article, we strengthen this multivariable result by proving that $\Theta_{1}$ is indeed a regular divisor of $\Theta'_1$, and we further establish the converse implication.

	%%%%%%%%%%%%%%%%%%%%%%%%%%%%%%%%%%%%%%%%%%%%%%%%%%%%%%
	
	\section{Preliminaries}
	
	We begin by establishing the setting of the full Fock space. Let \(\Gamma(\mathbb{C}^n)\), abbreviated as \(\Gamma\), denote the \emph{full Fock space} over \(\mathbb{C}^n\), defined as the direct sum \(\Gamma \coloneqq \bigoplus_{k \geq 0} (\mathbb{C}^n)^{\otimes k} = \mathbb{C} \oplus \mathbb{C}^n \oplus (\mathbb{C}^n \otimes \mathbb{C}^n) \oplus \cdots\). Let \(\mathbb{F}_n^+\) be the unital free semigroup generated by \(\{1, \cdots, n\}\) with identity \(\emptyset\). For a word \(\alpha = \alpha_1 \cdots \alpha_k \in \mathbb{F}_k^+\), we define the length \(|\alpha| = k\) (with \(|\emptyset| = 0\)) and the vector \(e_\alpha \coloneqq e_{\alpha_1} \otimes \cdots \otimes e_{\alpha_k}\), where \(e_\emptyset \coloneqq 1 \oplus 0 \oplus \cdots\) is the \emph{vacuum vector}. The set \(\{e_\alpha : \alpha \in \mathbb{F}_k^+\}\) forms an orthonormal basis for \(\Gamma\). For \( i = 1, \cdots, n \), the \emph{left} and \emph{right creation operators}, \( S_i, R_i : \Gamma \to \Gamma \), are defined respectively by \(S_i(x) \coloneqq e_i \otimes x\) and \(R_i(x) \coloneqq x \otimes e_i\) for all \(x \in \Gamma\).
	
	\begin{definition}
		Let \(\mathcal{H}\) be a Hilbert space. An \(n\)-tuple of operators \(T \coloneqq [T_1, \cdots, T_n]: \bigoplus_{i=1}^n\mathcal{H}\to \mathcal H\) is called a \emph{row contraction} if \(\sum_{i=1}^n T_i T_i^* \leq I_{\mathcal{H}}\). For \(\alpha = \alpha_1 \cdots \alpha_k \in \mathbb{F}_n^+\), we denote \(T_\alpha \coloneqq T_{\alpha_1} \cdots T_{\alpha_k}\), with \(T_\emptyset \coloneqq I_{\mathcal{H}}\). The associated \emph{defect operators} are defined as
		\[
		D_T \coloneqq (I - T^*T)^{1/2} \in B\left(\bigoplus_{i=1}^n \mathcal{H} \right) \quad \text{and} \quad D_{T^*} \coloneqq (I - TT^*)^{1/2} \in B(\mathcal{H}),
		\]
		with corresponding \emph{defect spaces} \(\mathcal{D}_T \coloneqq \overline{\im(D_T)}\) and \(\mathcal{D}_{T^*} \coloneqq \overline{\im(D_{T^*})}\).
		A row contraction \(T\) is said to be \emph{pure} if \(\lim_{k \to \infty} \sum_{|\alpha| = k} \|T_{\alpha}^* h\|^2 = 0\) for all \(h \in \mathcal{H}\).
	\end{definition}

	Central to this theory is the concept of multi-analytic operators. Let \(\mathcal{H}\) and \(\mathcal{K}\) be Hilbert spaces. A bounded linear map \(M : \Gamma \otimes \mathcal{H} \to \Gamma \otimes \mathcal{K}\) is called a \emph{multi-analytic operator} if it intertwines with the left creation operators, i.e., \(M(S_i \otimes I_{\mathcal{H}}) = (S_i \otimes I_{\mathcal{K}})M\) for \(i=1, \cdots, n\). Such an operator is uniquely determined by its \emph{symbol} \(\theta: \mathcal{H} \to \Gamma \otimes \mathcal{K}\), defined by \(\theta(h) \coloneqq M(e_\emptyset \otimes h)\). The operator with symbol $\theta$ is denoted by $M_\theta$. Two multi-analytic operators \(M\) and \(N\) \emph{coincide} if there exist unitary operators \(U, V\) such that \(N(I \otimes U) = (I \otimes V)M\).
	The operator \(M_\theta\) is classified as: \emph{inner} if it is an isometry; \emph{outer} if \(\overline{M_\theta(\Gamma \otimes \mathcal{H})} = \Gamma \otimes \mathcal{K}\); \emph{unitary constant} if \(M_\theta = I_\Gamma \otimes W\) for a unitary \(W\); and \emph{purely contractive} if \(\| P_{\mathcal{K}} \theta h \| < \| h \|\) for all nonzero \(h \in \mathcal{H}\), where \(P_{\mathcal{K}}\) projects onto \(e_\emptyset \otimes \mathcal{K}\).
	
	Following G. Popescu \cite{Po89b}, the \emph{characteristic function} of a row contraction \(T\) is the multi-analytic operator \(\Theta_{T} : \Gamma \otimes \mathcal{D}_{T} \to \Gamma \otimes \mathcal{D}_{T^*}\) with symbol \(\theta_{T}: \mathcal{D}_{T} \to \Gamma \otimes \mathcal{D}_{T^*}\) defined by:
	\[
	\theta_{T}(h) \coloneqq -\sum_{i=1}^n T_i P_i h + \sum_{i=1}^n (S_i \otimes I_{\mathcal{D}_{T^*}}) \left( \sum_{\alpha \in \mathbb{F}_n^+} e_{\alpha} \otimes D_{T^*} T^*_{\alpha} P_i D_{T} h \right), \quad h \in \mathcal{D}_{T},
	\]
	where \(P_i\) is the orthogonal projection of \(\bigoplus_{j=1}^{k}\mathcal{H}\) onto its \(i^{\text{th}}\) component.

	Let $\ttt $ be a row contraction acting on a Hilbert space $\mathcalH$. A tuple $\vvv $ of operators on a Hilbert space $\mathcal{K} \supseteq \mathcalH$ is called a \emph{minimal isometric dilation} of $\ttt$ if the following conditions are satisfied
	\begin{enumerate}[\rm (i)]
		\item The operators $V_1, \dots, V_n \in B(\mathcal{K})$ are isometries satisfying
		\[
		V_i^* V_j = \delta_{ij} I_{\mathcal{K}}, \quad \text{for all } 1 \leq i,j \leq n;
		\]
		
		\item For each $i \in \{1,\dots,n\}$, the compression satisfies
		\[
		T_i^* = V_i^*\vert_\mathcal{H};
		\]
		
		\item The space $\mathcal{K}$ is minimal in the sense that
		\[
		\mathcal{K} = \bigvee_{\alpha \in \mathbb F_n^+} V_\alpha \mathcalH.
		\]
	\end{enumerate}
	
	We define the defect spaces associated with this construction as
	\begin{equation}
		\mathcal{L} \coloneqq \bigvee_{i=1}^n (V_i - T_i)\mathcalH \quad \text{and} \quad \mathcal{L}_* \coloneqq \ov{\left(I_{\mathcal{K}} - \sum_{i=1}^n V_i T_i^*\right)\mathcalH}.
	\end{equation}
	For a wandering subspace $\mathcal E$ with respect to $V$, we define
	\[
	M_V(\mathcal{E}) \coloneqq \bigoplus_{\alpha \in \mathbb F_n^+} V_\alpha \mathcal{E}.
	\]
	As shown in \cite{Po89a}, the minimal isometric dilation  space $\mathcal{K}$ admits the orthogonal decompositions
	\begin{equation}\label{orthogonal_decomposition_K}
		\mathcal{K} = \mathcal R \oplus M_V(\mathcal{L}_*) 
		= \mathcalH \oplus M_V(\mathcal{L}),
	\end{equation}
	where
	\[
	\mathcal R = \bigcap_{m=0}^{\infty} \bigoplus_{|\alpha|=m} V_\alpha \mathcal{K}.
	\]
	If the row contraction is c.n.c., then we have
	\begin{equation}\label{c.n.c. conditions}
		\ov{M_V(\mathcal{L})\vee M_V(\mathcal{L}_*)}=\mathcal K.
	\end{equation}	
	Furthermore, let $\Phi^{\mathcal{L}} : M_V(\mathcal{L}) \to \Gamma \otimes \mathcal{L}$ be the canonical unitary representation defined by the mapping
	\begin{equation}
		\Phi^{\mathcal{L}} \left( \SA V_\alpha \ell_\alpha \right) \coloneq \SA e_\alpha \otimes \ell_\alpha,
	\end{equation}
	for $\{\ell_\alpha\}_{\alpha \in \mathbb F_n^+} \sub \mathcal{L}$ such that $\SA \|\ell_\alpha\|^2 < \infty$. This unitary operator serves as a Fourier representation, intertwining the dilation $V$ with the standard multi-shift $S = [S_1, \dots, S_n]$ via the identity
	\begin{equation}
		\Phi^{\mathcal{L}} V_i = (S_i \otimes I_{\mathcal{L}}) \Phi^{\mathcal{L}}, \quad i=1, \dots, n.
	\end{equation}

	%%%%%%%%%%%%%%%%%%%%%%%%%%%%%%%%%%%%%%%%%%%%%%%

	\section{$k$-Regular Factorizations of Multi-Analytic Operators}
	Let \(\Theta : \Gamma \otimes \mathcal{E} \to \Gamma \otimes \mathcal{E}_*\) be a contractive multi-analytic operator admitting a factorization \(\Theta = \Theta_k\cdots \Theta_1\), where  for each $i=1,\cdots,k,$ the operators \(\Theta_i : \Gamma \otimes \mathcal{E}_i \to \Gamma \otimes \mathcal{E}_{i+1}\) are contractive multi-analytic operators with $\mathcal E=\mathcal E_1$ and $\mathcal E_{k+1}=\mathcal E_*$. The associated defect operators are defined by
	\[
	\Delta_\Theta \coloneqq (I - \Theta^* \Theta)^{1/2}, \quad \Delta_i \coloneqq (I - \Theta_i^* \Theta_i)^{1/2}, \quad \Delta_{*i} \coloneqq (I - \Theta_i \Theta_i^*)^{1/2} \quad (i=1,\dots,k.).
	\]
	We associate with this factorization a linear isometry \(Z_k: \overline{\Delta_\Theta (\Gamma \otimes \mathcal{E})} \to \overline{\Delta_k (\Gamma \otimes \mathcal{E}_k)} \oplus\cdots\oplus \overline{\Delta_1 (\Gamma \otimes \mathcal{E}_1)}\) determined by
	\begin{equation}\label{k_regular}
		Z_k(\Delta_\Theta f) \coloneqq \Delta_k\Theta_{k-1}\cdots \Theta_1 f \oplus\cdots\oplus \Delta_1 f \quad \text{for all } f \in \Gamma \otimes \mathcal{E}.
	\end{equation}
	\begin{definition}\label{def_Z_k}
		The factorization $\Theta = \Theta_k \cdots \Theta_1$ is said to be an \emph{$k$-regular factorization} if $Z_k$ is unitary, that is,
		\[
		\overline{\{ \Delta_k\Theta_{k-1}\cdots \Theta_1 f \oplus\cdots\oplus \Delta_1 f  : f \in \Gamma \otimes \mathcal{E}_1 \}} 
		=
		\overline{\Delta_k(\Gamma \otimes \mathcal{E}_k)} \oplus\cdots\oplus \overline{\Delta_1 (\Gamma \otimes \mathcal{E}_1)}.
		\]
	\end{definition}
	
	Consider the factorization $\Theta = \Theta_k \cdots \Theta_1$ given in \eqref{k_regular}. The index set $\{1, \dots, k\}$ is partitioned into $r$ disjoint subsets $J_1, \dots, J_r$, defined by
	\[
	J_1 = \{ j_1, \dots, 1 \}, \quad \dots, \quad
	J_i = \{ j_i, \dots, j_{i-1} + 1 \}, \quad \dots, \quad
	J_r = \{ j_r, \dots, j_{r-1} + 1 \},
	\]
	where $1 \leq j_1 < j_2 < \dots < j_r = k$. Let $\Theta_{J_i}$ denote the product of the operators indexed by $J_i$, that is,
	\[
	\Theta_{J_1} \coloneqq \Theta_{j_1} \cdots \Theta_1, \quad \dots, \quad
	\Theta_{J_i} \coloneqq \Theta_{j_i} \cdots \Theta_{j_{i-1} + 1}, \quad \dots, \quad
	\Theta_{J_r} \coloneqq \Theta_k \cdots \Theta_{j_{r-1} + 1}.
	\]
	Since the operators $\Theta_j$ are assumed to be contractions, each operator $\Theta_{J_i}$ ($i = 1, \dots, r$) is also a contraction. According to Definition~\ref{def_Z_k}, the aggregated factorization
	\[
	\Theta= \Theta_{J_r} \cdots \Theta_{J_1}
	\]
	is called an $r$-regular factorization if the isometry associated with this partitioned factorization,
	\[
	Z_r^{J_r, \dots, J_1} \colon 
	\overline{\Delta_\Theta(\Gamma \otimes \mathcal{E}_1)}
	\to 
	\overline{\Delta_{\Theta_{J_r}}(\Gamma \otimes \mathcal{E}_{j_{r-1}+1})}
	\oplus \cdots \oplus
	\overline{\Delta_{\Theta_{J_1}}(\Gamma \otimes \mathcal{E}_1)},
	\]
	defined by
	\[
	Z_r^{J_r, \dots, J_1}(\Delta_\Theta f)
	\coloneqq
	\Delta_{\Theta_{J_r}} \Theta_{J_{r-1}} \cdots \Theta_{J_1} f
	\oplus \cdots \oplus
	\Delta_{\Theta_{J_1}} f, \quad f\in\Gamma \otimes \mathcal{E}_1
	\]
	is unitary. For notational convenience, when the partition consists entirely of singletons, that is,
	\(J_i=\{i\},~	i=1,\dots,k,\)	the associated isometry is denoted simply by \(Z_k\) instead of
	\(Z_k^{\{k\},\{k-1\},\dots,\{1\}}.\)
	Furthermore, the sub-factorization
	\[
	\Theta_{J_i}
	=
	\Theta_{j_i}\cdots\Theta_{j_{i-1}+1}
	\]
	is said to be \(|J_i|\)-regular if and only if the associated isometry
	\[
	Z_{|J_i|}^{\{j_i\},\dots,\{j_{i-1}+1\}}
	\colon
	\overline{\Delta_{\Theta_{J_i}}
		(\Gamma\otimes\mc E_{j_{i-1}+1})}
	\to
	\overline{\Delta_{\Theta_{j_i}}
		(\Gamma\otimes\mc E_{j_i})}
	\oplus \cdots \oplus
	\overline{\Delta_{\Theta_{j_{i-1}+1}}
		(\Gamma\otimes\mc E_{j_{i-1}+1})}
	\]
	defined by
	\begin{align*}
		&
		Z_{|J_i|}^{\{j_i\},\dots,\{j_{i-1}+1\}}
		(\Delta_{\Theta_{J_i}}x)
		\\
		&\qquad\coloneqq
		\Delta_{\Theta_{j_i}}
		\Theta_{j_i-1}\cdots\Theta_{j_{i-1}+1}x
		\oplus
		\Delta_{\Theta_{j_i-1}}
		\Theta_{j_i-2}\cdots\Theta_{j_{i-1}+1}x
		\oplus \cdots \oplus
		\Delta_{\Theta_{j_{i-1}+1}}x,
	\end{align*}
	for \(x\in \Gamma\otimes\mc E_{j_{i-1}+1}\), is unitary. In the trivial case where \(J_i=\{j_i\}\), the corresponding isometry
	\[
	Z_1^{\{j_i\}}
	\colon
	\overline{\Delta_{\Theta_{j_i}}
		(\Gamma\otimes\mc E_{j_i})}
	\to
	\overline{\Delta_{\Theta_{j_i}}
		(\Gamma\otimes\mc E_{j_i})}
	\]
	is understood to be the identity operator.

	Since contractive multi-analytic operators are, in particular, contractions, Propositions~2.2 and~2.3 of \cite{HM2026I} will be used repeatedly throughout the paper in the setting of contractive multi-analytic operators. For convenience, the corresponding analogues	 for contractive multi-analytic operators, which follow as particular cases of the corresponding propositions for contractions established in \cite{HM2026I}.

	\begin{proposition}\label{partition_multi}
		Let \(\Theta : \Gamma \otimes \mathcal{E} \to \Gamma \otimes \mathcal{E}_*\) be a contractive multi-analytic operator admitting a factorization \(\Theta = \Theta_k \cdots \Theta_1\), where, for $i=1,\dots,k$ with $k \geq 2$, the operators \(\Theta_i : \Gamma \otimes \mathcal{E}_i \to \Gamma \otimes \mathcal{E}_{i+1}\) are contractive multi-analytic operators satisfying $\mathcal E=\mathcal E_1$ and $\mathcal E_{k+1}=\mathcal E_*$. The following statements are equivalent
		\begin{enumerate}[\rm (i)]
			\item The factorization \(\Theta = \Theta_k \cdots \Theta_1\) is $k$-regular.
			
			\item For every disjoint partition $J_1, \dots, J_r$ of $\{1, \dots, k\}$, the factorization $\Theta = \Theta_{J_r} \cdots \Theta_{J_1}$ is $r$-regular, and for each $i \in \{1, \dots, r\}$, the sub-factorization $\Theta_{J_i} = \Theta_{j_i} \cdots \Theta_{j_{i-1} + 1}$ is $|J_i|$-regular.
			
			\item There exists a disjoint partition $J_1, \dots, J_r$ of $\{1, \dots, k\}$ such that the factorization $\Theta = \Theta_{J_r} \cdots \Theta_{J_1}$ is $r$-regular, and for each $i \in \{1, \dots, r\}$, the sub-factorization $\Theta_{J_i} = \Theta_{j_i} \cdots \Theta_{j_{i-1} + 1}$ is $|J_i|$-regular.
		\end{enumerate}
	\end{proposition}
	
	\begin{proof}
		
		The proof follows from the identity established in the proof of Proposition 2.2 in \cite{HM2026I}
		\begin{equation}\label{deco_relation_ZK}
			Z_k = \left( \bigoplus_{i=r}^{1} Z_{|J_i|}^{\{j_i\}, \dots, \{j_{i-1}+1\}} \right) Z_r^{J_r, \dots, J_1},
		\end{equation}
		where $Z_k$, $Z_{|J_i|}^{\{j_i\}, \dots, \{j_{i-1}+1\}}$, and $Z_r^{J_r, \dots, J_1}$ are the isometries associated with the factorizations
		\[
		\Theta = \Theta_k \cdots \Theta_1, 
		\quad 
		\Theta_{J_i} := \Theta_{j_i} \cdots \Theta_{j_{i-1}+1}, 
		\quad 
		\Theta = \Theta_{J_r} \cdots \Theta_{J_1},
		\]
		respectively.
	\end{proof}

	\begin{proposition}\label{propo_k_regular_multi}
		Let \(\Theta : \Gamma \otimes \mathcal{E} \to \Gamma \otimes \mathcal{E}_*\) be a contractive multi-analytic operator admitting a factorization \(\Theta = \Theta_k \cdots \Theta_1\), where, for $i=1,\dots,k$ with $k \geq 2$, the operators \(\Theta_i : \Gamma \otimes \mathcal{E}_i \to \Gamma \otimes \mathcal{E}_{i+1}\) are contractive multi-analytic operators satisfying $\mathcal E=\mathcal E_1$ and $\mathcal E_{k+1}=\mathcal E_*$. Then the following statements are equivalent
		\begin{enumerate}[label=\rm(\roman*), ref=\thetheorem(\roman*)]
			\item The factorization \( \Theta = \Theta_k \cdots \Theta_1 \) is a \( k \)-regular factorization.
			
			\item The factorization \( \Theta = \Theta_k(\Theta_{k-1} \cdots \Theta_1) \) is a $2$-regular factorization, and \( \Theta_{k-1} \cdots \Theta_1 \) is a \( (k-1) \)-regular factorization.
			
			\item The factorization \( \Theta_k \cdots \Theta_2 \) is a \( (k-1) \)-regular factorization, and the factorization \( \Theta = (\Theta_k \cdots \Theta_2) \Theta_1 \) is a $2$-regular factorization.
			
			\item The factorizations $(\Theta_k \cdots \Theta_{j+1})(\Theta_j \cdots \Theta_1)$ are $2$-regular factorizations for all $j = 1, \dots, k-1$.
		\end{enumerate}
	\end{proposition}
	\begin{proof}
		The proof is omitted, since it proceeds analogously to the proof of Proposition~2.3 in \cite{HM2026I}.
	\end{proof}
	%%%%%%%%%%%%%%%%%%%%%%%%%%%%%%%%%%%%

	To discuss \(k\)-regular factorizations and joint invariant subspaces, we require a generalized version of Lemma 3.1 from \cite{Po06}. The result is stated below; its proof follows by a straightforward adaptation of the arguments provided in the proof of Lemma 3.1 in \cite{Po06}.   
	\begin{lemma}\label{cuntz_isometry}
		Let \(\Theta: \Gamma \otimes \mathcal{E} \to \Gamma \otimes \mathcal{E}_*\) be a contractive multi-analytic operator. Let \(C \coloneqq [C_1, \ldots, C_n]\) be the row isometry defined on the subspace \( \overline{\Delta_\Theta(\Gamma \otimes \mathcal{E})}\) by
		\[
		C_j \Delta_\Theta f \coloneqq \Delta_\Theta (S_j \otimes I_{\mathcal{E}}) f, \quad f \in \Gamma \otimes \mathcal{E},
		\]
		for each \(j = 1, \ldots, n\), where \(\Delta_\Theta \coloneqq (I - \Theta^* \Theta)^{1/2}\). Then \(C\) is a Cuntz row isometry (i.e., \(\sum_{j=1}^n C_j C_j^* = I_\mathcal{K}\)) if and only if
		\begin{equation}\label{szego_condition}
			\overline{\Delta_\Theta(\Gamma \otimes \mathcal{E})} = \overline{\Delta_\Theta((\Gamma \otimes \mathcal{E}) \ominus \mathcal{E})}.
		\end{equation}
		
		Suppose that \(\Theta\) admits a factorization
		\[
		\Theta = \Theta_k \Theta_{k-1} \cdots \Theta_1,
		\]
		where \(\Theta_i : \Gamma \otimes \mathcal{E}_i \to \Gamma \otimes \mathcal{E}_{i+1}\) are contractive multi-analytic operators for \(i=1, \cdots, k\), with \(\mathcal{E}_1 = \mathcal{E}\) and \(\mathcal{E}_{k+1} = \mathcal{E}_*\). Let \(C^{(i)} \coloneqq [C_1^{(i)}, \ldots, C_n^{(i)}]\) denote the row isometry defined on \(\overline{\Delta_i(\Gamma \otimes \mathcal{E}_i)}\) associated with \(\Theta_i\), where \(\Delta_i \coloneqq (I - \Theta_i^* \Theta_i)^{1/2}\).
		Then the following intertwining relation holds
		\begin{equation}\label{Int-XT-gen}
			Z_k C_j = \begin{pmatrix}
				C_j^{(k)} & 0 & \cdots & 0 \\
				0 & C_j^{(k-1)} & \cdots & 0 \\
				\vdots & \vdots & \ddots & \vdots \\
				0 & 0 & \cdots & C_j^{(1)}
			\end{pmatrix} Z_k, \quad j = 1, \ldots, n,
		\end{equation}
		where \(Z_k\) is the isometric map associated with the factorization \(\Theta = \Theta_k \cdots \Theta_1\). Furthermore, if the factorization is $k$-regular, then \(C\) is a Cuntz row isometry if and only if \(C^{(i)}\) is a Cuntz row isometry for every \(i = 1, \ldots, k\).
	\end{lemma}
	
	First, let us recall that for \( k=2 \), G. Popescu, in the article \cite{Po06}, established a one-to-one correspondence between the joint invariant subspaces of c.n.c. row contractions and $2$-regular factorizations of characteristic functions. He also developed a functional model for a given $2$-regular factorization, described as follows
	\begin{theorem}{ \rm (Theorem 3.2, \cite{Po06})}\label{2_r_invariant}
		Let $T\coloneqq[T_1,\cdots,T_n]$ be a c.n.c. row contraction on a Hilbert space $\mathcal{H}$, and let $\Theta: \Gamma \otimes \mathcal{E} \to \Gamma \otimes \mathcal{E}_*$ be a contractive multi-analytic operator that coincides with the characteristic function of $T$.
		
		If $\mathcal{H}_1 \sub \mathcal{H}$ is a joint invariant subspace for $T_1,\ldots, T_n$, then there exists a $2$-regular factorization $\Theta = \Theta_2 \Theta_1$, where $\Theta_1: \Gamma \otimes \mathcal{E} \to \Gamma \otimes \mathcal{F}$ and $\Theta_2: \Gamma \otimes \mathcal{F} \to \Gamma \otimes \mathcal{E}_*$ are contractive multi-analytic operators, such that $T$ is unitarily equivalent to a row contraction $\wh{T} \coloneqq [\wh{T}_1,\ldots, \wh{T}_n]$ defined on the model space
		\begin{align*}
			\wh{\mathcal H} \coloneqq \bigg[ (\Gamma \otimes \mathcal{E}_*) & \oplus \overline{\Delta_2 (\Gamma \otimes \mathcal{F})} \oplus \overline{\Delta_1 (\Gamma \otimes \mathcal{E})} \bigg] \\
			& \ominus \big\{ \Theta_2 \Theta_1 f \oplus \Delta_2 \Theta_1 f \oplus \Delta_1 f : f \in \Gamma \otimes \mathcal{E} \big\}.
		\end{align*}
		For each $j=1, \cdots, n$, the action of the adjoint operators $\wh{T}_j^*$ on the 
		Hilbert space $\wh{\mathcal H}$ is defined by
		\[
		\wh{T}_j^*(f \oplus \varphi \oplus \psi) \coloneqq (S_j^* \otimes I_{\mathcal{E}_*}) f \oplus C_j^{(2)*} \varphi \oplus C_j^{(1)*} \psi, \quad (f \oplus \varphi \oplus \psi \in \wh{\mathcal H}),
		\]
		where $S_1, \cdots, S_n$ denote the left creation operators on the free Fock space $\Gamma$, while $C^{(1)} = [C_1^{(1)}, \cdots, C_n^{(1)}]$ and $C^{(2)} = [C_1^{(2)}, \cdots, C_n^{(2)}]$ are the row isometries associated with $\Theta_1$ and $\Theta_2$ respectively, as defined in Lemma $\ref{cuntz_isometry}$.
		Moreover, the subspaces of $\wh{\mathcal H}$ corresponding to $\mathcal{H}_1$ and $\mathcal{H}_2 \coloneqq \mathcal{H} \ominus \mathcal{H}_1$ are given by
		\begin{align*}
			\wh{\mathcal H}_1 \coloneqq \bigg\{ \Theta_2 g \oplus \Delta_2 g \oplus h : \ & g \in \Gamma \otimes \mathcal{F}, \ h \in \overline{\Delta_1 (\Gamma \otimes \mathcal{E})} \bigg\} \\
			& \ominus \big\{ \Theta_2 \Theta_1 f \oplus \Delta_2 \Theta_1 f \oplus \Delta_1 f : f \in \Gamma \otimes \mathcal{E} \big\},
		\end{align*}
		and
		\begin{align*}
			\wh{\mathcal H}_2 \coloneqq \bigg[ (\Gamma \otimes \mathcal{E}_*) & \oplus \overline{\Delta_2 (\Gamma \otimes \mathcal{F})} \oplus \{0\} \bigg] \\
			& \ominus \big\{ \Theta_2 g \oplus \Delta_2 g \oplus \{0\} : g \in \Gamma \otimes \mathcal{F} \big\},
		\end{align*}
		respectively.
		
		Conversely, every $2$-regular factorization $\Theta = \Theta_2 \Theta_1$ generates, via the formulas above, the subspaces $\wh{\mathcal H}_1$ and $\wh{\mathcal H}_2$ satisfying
		\begin{enumerate}[\rm (i)]
			\item $\wh{\mathcal H}_1$ is invariant under each operator $\wh{T}_j$, for $j=1,\ldots,n$;
			\item$\wh{\mathcal H}_2 = \wh{\mathcal H} \ominus \wh{\mathcal H}_1$.
		\end{enumerate}
		Under this identification, $\wh{\mathcal H}_1$ corresponds to a subspace $\mathcal{H}_1 \sub \mathcal{H}$ which is invariant under each operator $T_i$, $j=1,\ldots,n$.
	\end{theorem}
	\section{Main Theorem}
	For $k \geq 2$, we state the main theorem, which establishes a correspondence between chain of joint invariant subspaces of a c.n.c. row contraction and the $k$-regular factorizations of its characteristic function. Moreover, the theorem proved by G. Popescu arises as a special case of this result when $k = 2$.

	\begin{theorem}\label{main_thereom_02}
		Let $\ttt$ be a c.n.c. row contraction on the Hilbert space $\mathcal H.$ Let $\wt{\Theta}:\Gamma\otimes\mathcal E\to\Gamma\otimes\mathcal E_* $ is a contractive multi-analytic operator such that it is coincide with the characteristic function of $T.$ Then the operator \( T \) is unitarily equivalent to the operator \( \wt{T} \), which is defined on the Hilbert space 
		\begin{align*}
			\wt{\mathcal H} \coloneqq \bigg[ (\Gamma \otimes \mathcal{E}_{*}) & \oplus \overline{\wt \Delta (\Gamma \otimes \mathcal{E})}\bigg]  \ominus \wt {\mathcal G}.
		\end{align*}
		The action of the adjoint operators is given by
		\[
		\wt{T}_j^*(f \oplus  g) \coloneqq (S_j^* \otimes I_{\mathcal{E}_{*}}) f \oplus \wt C_j^{*} g, \quad (f \oplus g\in \wt{\mathcal H}), \text{ for } j=1,\cdots,n
		\]
		here, the subspace $\wt{\mathcal G}$ is defined as
		\[
		\wt{\mathcal G} \coloneqq \{ \wt{\Theta} u \oplus \wt  \Delta u : u \in \Gamma\otimes\mathcal{E}\}, 
		\]
		where $\wt \Delta \coloneqq (I - \wt{\Theta}^* \wt{\Theta})^{1/2}$ is the defect operator, $S_1, \cdots, S_n$ are the left creation operators on the full Fock space $\Gamma$, and $\wt C = [\wt C_1, \cdots, \wt C_n]$ is the row isometries associated with $\wt{\Theta}$,  as defined in Lemma $ \ref{cuntz_isometry}$.
		
		Assume further that \( k \geq 2 \) and that 
		$$\mathcal M_1\sub\mathcal M_2\sub\cdots\sub\mathcal M_{k-1}$$ are joint invariant subspaces  of $T,$ then there exists an \(k\)-regular factorization of the form 
		\begin{equation}\label{m_regular_fact}
			\wt{\Theta} = \wt{\Theta}_k \cdots \wt{\Theta}_1,
		\end{equation}
		where \( \wt{\Theta}_i:\Gamma\otimes \mathcal{E}_i\to \Gamma\otimes \mathcal{E}_{i+1}  \) are contractive multi-analytic operators for \( i=1,\cdots,k \) with \( \mathcal{E}_1 = \mathcal{E} \) and \( \mathcal{E}_{k+1} = \mathcal{E}_* \).
		Additionally, for \( i = 1, \dots, k-1 \), the joint invariant subspaces $\mathcal M_i$ correspond to
		\begin{align*}
			\wt{\mathcal M}_i
			=
			\Big\{
			&\wt{\Theta}_k\cdots\wt{\Theta}_{i+1}u_{i+1} \oplus
			\wt Z_k^*\big(
			\wt\Delta_k\wt{\Theta}_{k-1}\cdots\wt{\Theta}_{i+1}u_{i+1}
			\oplus\cdots\oplus
			\wt\Delta_{i+1}u_{i+1}
			\oplus v_i\oplus\cdots\oplus v_1
			\big)\\
			&\quad\quad : u_{i+1}\in \Gamma\otimes \mathcal E_{i+1},v_j\in \ov{\im\wt \Delta_j} , j=1,\dots,i
			\Big\}
			\ominus\wt{\mathcal G},
		\end{align*}
		Moreover,  the orthogonal complement \( \wt{\mathcal N}_i = \wt{\mathcal H} \ominus \wt{\mathcal M}_i \) is given by
		\begin{align*}
			\wt { \mathcal N}_i&= \left[ \Gamma\otimes\mathcal{E}_* \oplus \wt{Z}_k^{*} \left( \ov{\wt \Delta_k (\Gamma\otimes\mathcal{E}_{k})} \oplus \cdots \oplus \ov{\wt \Delta_{i+1} (\Gamma\otimes\mathcal{E}_{i+1})} \oplus \{0\} \oplus \cdots \oplus \{0\} \right) \right] \\
			&\quad \ominus \left\{ \wt {\Theta}_k \cdots \wt {\Theta}_{i+1} u_{i+1} + \wt{Z}_k^{*} ( \wt \Delta_k \wt {\Theta}_{k-1} \cdots \wt {\Theta}_{i+1} u_{i+1} \oplus \cdots \oplus \wt \Delta_{i+1} u_{i+1} \oplus 0 \oplus \cdots \oplus 0 )  \right. \\
			&  \qquad \left. : u_{i+1} \in \Gamma\otimes\mathcal{E}_{i+1} \right\}. \\
		\end{align*}
		Here
		\[
		\wt{Z}_k :
		\overline{\wt{\Delta} (\Gamma\otimes\mathcal{E})}
		\longrightarrow
		\overline{\wt{\Delta}_k (\Gamma\otimes\mathcal{E}_k)}
		\oplus \cdots \oplus
		\overline{\wt{\Delta}_1 (\Gamma\otimes\mathcal{E}_1)}
		\]
		is the isometry defined by
		\begin{equation*}
			\wt{Z}_k \wt{\Delta} v
			=
			\wt{\Delta}_k \wt{\Theta}_{k-1} \cdots \wt{\Theta}_1 v
			\oplus \cdots \oplus
			\wt{\Delta}_2 \wt{\Theta}_1 v
			\oplus
			\wt{\Delta}_1 v,
			\qquad
			v \in \Gamma\otimes\mathcal{E}.
		\end{equation*}	
		Conversely, if
		\[
		\wt{\Theta}
		=
		\wt{\Theta}_k\cdots\wt{\Theta}_1
		\]
		is a $k$-regular factorization and the subspaces 
		$\wt{\mathcal M}_i$ and $\wt{\mathcal N}_i$ are defined as above, 
		then each $\wt{\mathcal M}_i$ is joint invariant under $\wt{T}$, the collection $\{\wt{\mathcal M}_i\}_{i=1}^{k-1}$ forms an increasing chain
		\[
		\wt{\mathcal M}_1
		\sub\cdots\sub
		\wt{\mathcal M}_{k-1},
		\]
		and for every $i=1,\dots,k-1$ we have the orthogonal decomposition
		\[
		\wt{\mathcal H}
		=
		\wt{\mathcal M}_i
		\oplus
		\wt{\mathcal N}_i.
		\]
	\end{theorem}

	%%%%%%%%%%%%%%%%%%%%%%%%%%%%%%%%%%%%%%%%%%%%%%%%%%%%%%%%%%%%%%%%%%%%%%%
	
	\begin{proof}
		Let $\mathcal M_{1} \sub \cdots \sub \mathcal M_{k-1}$ be joint invariant subspaces for a c.n.c. row contraction $T = [T_1, \dots, T_n]$ acting on $\mathcal{H}$. Assume that $V = [V_1, \dots, V_n]$ is the minimal isometric dilation of $T$ on the Hilbert space $\mathcal{K}$.
		By Popescu's Wold decomposition (see \cite{Po89a}), one has
		\begin{equation*}
			\mathcal{K} = \mathcal{R} \oplus M_V(\mathcal{L}_*),
		\end{equation*}
		where $\mathcal{R} = \bigcap_{m=0}^{\infty} \bigoplus_{|\alpha|=m} V_\alpha \mathcal{K}$ is the maximal reducing subspace for the operators $V_1, \dots, V_n$ such that $V|_{\mathcal{R}}$ is a Cuntz row isometry, and $\mathcal{L}_* = \overline{(I_{\mathcal{H}} - \sum_{j=1}^{n} V_j T_j^*) \mathcal{H}}.$
		
		For $i = 1, \dots, k-1$, define $\mathcal{N}_i \coloneqq \mathcal{H} \ominus \mathcal M_i$. Clearly, each $\mathcal{N}_i$ is invariant under $T_j^*$ for $j = 1, \dots, n$. Consequently, each $\mathcal{N}_i$ is also invariant under $V_j^*$, since $T_j^* = V_j^*|_{\mathcal{H}}$. This implies that the orthogonal complements $\mathcal{K}_i \coloneqq \mathcal{K} \ominus \mathcal{N}_i$ are invariant under $V_j$. Applying  Popescu’s Wold decomposition to the restricted row isometry $[V_1|_{\mathcal{K}_i}, \dots, V_n|_{\mathcal{K}_i}]$, we obtain
		\begin{equation*}
			\mathcal{K}_i = M_V(\mathcal{F}_i) \oplus \mathcal{R}_i,
		\end{equation*}
		where $$\mathcal{R}_i = \bigcap_{m=0}^{\infty} \bigoplus_{|\alpha|=m} V_\alpha \mathcal{K}_i, ~~\mathcal{F}_i = \mathcal{K}_i \ominus \left( \bigoplus_{j=1}^n V_j \mathcal{K}_i \right).$$ Moreover, $[V_1|_{M_V(\mathcal{F}_i)}, \dots, V_n|_{M_V(\mathcal{F}_i)}]$ is a multi-shift with wondering subspace $\mathcal F_i$, and $\mathcal{R}_i$ is the maximal subspace of $\mathcal{K}_i$ reducing the operators $V_1, \dots, V_n$ on which the row isometry acts as a Cuntz row isometry.
		
		Observe that the inclusions $\mathcal M_1 \sub \cdots \sub \mathcal M_{k-1}$ naturally imply $\mathcal{N}_1 \supseteq \mathcal{N}_2 \supseteq \cdots \supseteq \mathcal{N}_{k-1}$, which therefore implies $\mathcal{K}_1 \sub \mathcal{K}_2 \sub \cdots \sub \mathcal{K}_{k-1}$. Therefore,
		\begin{equation}\label{k_subset k_{i+1}}
			M_V(\mathcal{F}_i) \oplus \mathcal{R}_i \sub M_V(\mathcal{F}_{i+1}) \oplus \mathcal{R}_{i+1} \quad \text{for } i = 1, \dots, k-2.
		\end{equation}
		We claim that  $\mathcal{R}_i \sub \mathcal{R}_{i+1}$ because $\mathcal{K}_i \sub \mathcal{K}_{i+1}$ and  the subspaces $\mathcal{R}_i$ and $\mathcal{R}_{i+1}$ are the maximal reducing subspaces on which the respective restricted isometries act as Cuntz row isometries, the maximality of $\mathcal{R}_{i+1}$ enforces this precise inclusion. Hence, one obtains the chain
		\begin{equation*}
			\mathcal{R}_1 \sub \mathcal{R}_2 \sub \cdots \sub \mathcal{R}_{k-1}.
		\end{equation*}
		The orthogonal difference spaces are defined by $\mathcal{W}_1 \coloneqq \mathcal{R}_1$, $\mathcal{W}_i \coloneqq \mathcal{R}_i \ominus \mathcal{R}_{i-1}$ for $i = 2, \dots, k-1$, and $\mathcal{W}_k \coloneqq \mathcal{R} \ominus \mathcal{R}_{k-1}$. It follows immediately that
		\begin{equation*}
			\mathcal{R} = \mathcal{W}_k \oplus \cdots \oplus \mathcal{W}_1.
		\end{equation*}
		Consequently, an application of equation~\eqref{k_subset k_{i+1}} yields
		\begin{equation}\label{m(f_i)}
			M_V(\mathcal{F}_i) \sub M_V(\mathcal{F}_{i+1}) \oplus \mathcal{W}_{i+1} \quad \text{for } i = 1, \dots, k-2.
		\end{equation}
		Recall the fundamental geometric orthogonal decompositions of $\mathcal{K}$
		\begin{equation*}
			\mathcal{K} = \mathcal{H} \oplus M_V(\mathcal{L}) = M_V(\mathcal{L}_*) \oplus \mathcal{R}.
		\end{equation*}
		Using the first decomposition alongside $\mathcal{K}_i = M_V(\mathcal{F}_i) \oplus \mathcal{R}_i$, we obtain
		\begin{align}
			[\mathcal{H} \oplus M_V(\mathcal{L})] \ominus [\mathcal{H} \ominus \mathcal M_i] = M_V(\mathcal{F}_i) \oplus \mathcal{R}_i.\\
			\mathcal M_i \oplus M_V(\mathcal{L}) = M_V(\mathcal{F}_i) \oplus \mathcal{R}_i.\label{M_i}
		\end{align}
		Which specifically yields
		\begin{equation}\label{m(l)}
			M_V(\mathcal{L}) \sub M_V(\mathcal{F}_i) \oplus \mathcal{R}_i~~\text{ for} ~~i=1,\dots, k-1.
		\end{equation}
		%%%%%%%%%%%%%%%%%%%%%%%%%%%%%%%%%%%%%%%%%%%%%%%%
		Using the second geometric decomposition of $\mathcal{K}$, we similarly get
		\begin{equation*}
			[M_V(\mathcal{L}_*) \oplus \mathcal{R}] \ominus [\mathcal{H} \ominus \mathcal M_{i}] = M_V(\mathcal{F}_{i}) \oplus \mathcal{R}_{i}.
		\end{equation*}
		This equivalence yields $M_V(\mathcal{F}_{i}) = [M_V(\mathcal{L}_*) \oplus (\mathcal{R} \ominus \mathcal{R}_{i})] \ominus [\mathcal{H} \ominus \mathcal M_{i}]$, from which we conclude
		\begin{equation}\label{m{f_{k-1}}}
			M_V(\mathcal{F}_{i}) \sub M_V(\mathcal{L}_*) \oplus (\mathcal{R} \ominus \mathcal{R}_{i}) ~~\text{ for} ~~i=1,\dots, k-1.
		\end{equation}
		From equations \eqref{m(l)},\eqref{m(f_i)} and \eqref{m{f_{k-1}}}, the following chain of inclusions holds
		\begin{align}
			M_V(\mathcal{L})& \sub M_V(\mathcal{F}_1) \oplus \mathcal{W}_1\label{L}\\
			M_V(\mathcal{F}_i)& \sub M_V(\mathcal{F}_{i+1}) \oplus \mathcal{W}_{i+1} \quad \text{for } i = 1, \dots, k-2 \label{f_i}\\ 
			M_V(\mathcal{F}_{k-1}) &\sub M_V(\mathcal{L}_*) \oplus \mathcal{W}_k. \label{f_k-1}
		\end{align}
		
		Let $P^{\mathcal{L}}, P^{\mathcal{L}_*}, P^{\mathcal{F}_i}$ (for $i = 1, \dots, k-1$), $P_{\mathcal{R}}$, and $P_{\mathcal{W}_i}$ (for $i = 1, \dots, k$) denote the orthogonal projections from $\mathcal{K}$ onto $M_V(\mathcal{L}), M_V(\mathcal{L}_*), M_V(\mathcal{F}_i), \mathcal{R},$ and $\mathcal{W}_i$, respectively.
		For an arbitrary vector $l \in M_V(\mathcal{L})$, using equation \eqref{L}, we can write $l = P^{\mathcal{F}_1}l \oplus P_{\mathcal{W}_1}l$. Iterating this orthogonal decomposition through the chain of inclusions given in equation \eqref{f_i} and \eqref{f_k-1}, one deduces
		\begin{align*}
			l &= P^{\mathcal{F}_2} P^{\mathcal{F}_1} l \oplus P_{\mathcal{W}_2} P^{\mathcal{F}_1} l \oplus P_{\mathcal{W}_1} l \nonumber \\
			&= P^{\mathcal{F}_3} P^{\mathcal{F}_2} P^{\mathcal{F}_1} l \oplus P_{\mathcal{W}_3} P^{\mathcal{F}_2} P^{\mathcal{F}_1} l \oplus P_{\mathcal{W}_2} P^{\mathcal{F}_1} l \oplus P_{\mathcal{W}_1} l \nonumber \\
			&\;\;\vdots \nonumber \\
			&= P^{\mathcal{L}_*} P^{\mathcal{F}_{k-1}} \cdots P^{\mathcal{F}_1} l \oplus P_{\mathcal{W}_k} P^{\mathcal{F}_{k-1}} \cdots P^{\mathcal{F}_1} l \oplus \cdots \oplus P_{\mathcal{W}_1} l.
		\end{align*}
		Consequently, the following explicit expressions for the projections onto $M_V(\mathcal{L}_*)$ and $\mathcal{R}$  restricted to $M_V(\mathcal{L})$ are obtained
		\begin{align}
			P^{\mathcal{L}_*} l &= P^{\mathcal{L}_*} P^{\mathcal{F}_{k-1}} \cdots P^{\mathcal{F}_1} l,\label{proj:L_*} \\
			P_{\mathcal{R}} l &= P_{\mathcal{W}_k} P^{\mathcal{F}_{k-1}} \cdots P^{\mathcal{F}_1} l \oplus \cdots \oplus P_{\mathcal{W}_2} P^{\mathcal{F}_1} l \oplus P_{\mathcal{W}_1} l.\label{proj:R}
		\end{align}
		As a direct consequence of \eqref{orthogonal_decomposition_K} and \eqref{L}–\eqref{f_k-1}, we have
		\begin{align}
			P_{\mathcal R} x &= (I - P^{\mathcal{L}_*}) x 
			&& \text{for } x \in \mathcal{K}, 
			\label{eq:PR} \\
			P_{\mathcal{W}_1} l &= (I - P^{\mathcal{F}_1}) l 
			&& \text{for } l \in M_V(\mathcal{L}), 
			\label{eq:PW1} \\
			P_{\mathcal{W}_{i+1}} f_i &= (I - P^{\mathcal{F}_{i+1}}) f_i 
			&& \text{for } f_i \in M_V(\mathcal{F}_i), \ (i = 1, \dots, k-2), 
			\label{eq:PWi} \\
			P_{\mathcal{W}_k} f_{k-1} &= (I - P^{\mathcal{L}_*}) f_{k-1} 
			&& \text{for } f_{k-1} \in M_V(\mathcal{F}_{k-1}).
			\label{eq:PWk}
		\end{align}
		Since $T$ is a c.n.c. row contraction, it follows from \eqref{c.n.c. conditions} that
		\begin{equation*}
			\overline{P_{\mathcal R} M_V(\mathcal L)} = \mathcal R.\label{P_RM(L)}
		\end{equation*}
		Combining this with \eqref{proj:R}, we obtain
		\begin{equation*}
			\overline{P_{\mathcal{W}_k} P^{\mathcal{F}_{k-1}} \cdots P^{\mathcal{F}_1} M_V(\mathcal{L})} \oplus \cdots \oplus \overline{P_{\mathcal{W}_1} M_V(\mathcal{L})} = \mathcal{W}_k \oplus \cdots \oplus \mathcal{W}_1.
		\end{equation*}
		Consequently, by matching the corresponding direct summands
		\begin{equation*}
			\overline{P_{\mathcal{W}_1} M_V(\mathcal{L})} = \mathcal{W}_1.\label{P_{W_1}M(L)}
		\end{equation*}
		From the inclusion 
		$M_V(\mathcal L) \sub M_V(\mathcal F_1) \oplus \mathcal W_1,$
		we obtain
		$\overline{P_{\mathcal W_2} P^{\mathcal F_1} M_V(\mathcal L)}
		\sub
		\overline{P_{\mathcal W_2} M_V(\mathcal F_1)}
		\sub
		\mathcal W_2.$
		Consequently,
		\begin{equation*}
			\overline{P_{\mathcal W_2} M_V(\mathcal F_1)} = \mathcal W_2.\label{P_{W_2}M(F_1)}
		\end{equation*}
		By an inductive reasoning analogous to the above, we completely determine that
		\begin{align}
			\overline{P_{\mathcal{W}_{i+1}} M_V(\mathcal{F}_i)} &= \mathcal{W}_{i+1} \quad \text{for } i = 1, \dots, k-2,  \label{P_{W_i+1}M(F_i)}\\
			\overline{P_{\mathcal{W}_k} M_V(\mathcal{F}_{k-1})} &= \mathcal{W}_k.\label{P_{W_k}M(F_k-1)}
		\end{align}
		Now, consider the following contractions
		\begin{align*}
			Q &\coloneqq P^{\mathcal{L}_*}|_{M_V(\mathcal{L})} : M_V(\mathcal{L}) \to M_V(\mathcal{L}_*), \\
			Q_1 &\coloneqq P^{\mathcal{F}_1}|_{M_V(\mathcal{L})} : M_V(\mathcal{L}) \to M_V(\mathcal{F}_1), \\
			Q_{i+1} &\coloneqq P^{\mathcal{F}_{i+1}}|_{M_V(\mathcal{F}_i)} : M_V(\mathcal{F}_i) \to M_V(\mathcal{F}_{i+1}) \quad \text{for } i = 1, \dots, k-2, \\
			Q_k &\coloneqq P^{\mathcal{L}_*}|_{M_V(\mathcal{F}_{k-1})} : M_V(\mathcal{F}_{k-1}) \to M_V(\mathcal{L}_*).
		\end{align*}
		Since $M_V(\mathcal{L})$, $M_V(\mathcal{L}_*)$, and $M_V(\mathcal{F}_i)$ are reducing subspaces for the operators $V_1, \dots, V_n$, the following intertwining relations for each $j = 1, \dots, n$
		\begin{align}
			Q (V_j|_{M_V(\mathcal{L})}) &= (V_j|_{M_V(\mathcal{L}_*)}) Q,\label{eq:q} \\
			Q_1 (V_j|_{M_V(\mathcal{L})}) &= (V_j|_{M_V(\mathcal{F}_1)}) Q_1,\label{eq:q1} \\
			Q_{i+1} (V_j|_{M_V(\mathcal{F}_i)}) &= (V_j|_{M_V(\mathcal{F}_{i+1})}) Q_{i+1},\label{eq:q_i+1} \\
			Q_k (V_j|_{M_V(\mathcal{F}_{k-1})}) &= (V_j|_{M_V(\mathcal{L}_*)}) Q_k.\label{eq:q_k}
		\end{align}
		
		Let 
		$\Phi^{\mathcal L} : M_V(\mathcal L) \longrightarrow \Gamma \otimes \mathcal L$
		denote the canonical Fourier unitary representation defined by
		\begin{equation}\label{intertwining_phi }
			\Phi^{\mathcal L}\!\left( \sum_{\alpha \in \mathbb F_n^+} V_\alpha \ell_\alpha \right)
			\coloneqq
			\sum_{\alpha \in \mathbb F_n^+} e_\alpha \otimes \ell_\alpha,
		\end{equation}
		for $\{\ell_\alpha\}_{\alpha \in \mathbb F_n^+} \sub \mathcal L$ satisfying 
		\[
		\sum_{\alpha \in \mathbb F_n^+} \|\ell_\alpha\|^2 < \infty.
		\]
		This unitary operator intertwines the dilation $V$ with the standard multi-shift 
		$S = [S_1, \dots, S_n]$ in the sense that
		\begin{equation*}
			\Phi^{\mathcal L} V_i
			=
			(S_i \otimes I_{\mathcal L}) \Phi^{\mathcal L},
			\quad i = 1, \dots, n.
		\end{equation*}
		Similarly, one defines the Fourier representations 
		$\Phi^{\mathcal L_*}$ and $\Phi^{\mathcal F_i}$ associated with 
		$M_V(\mathcal L_*)$ and $M_V(\mathcal F_i)$, respectively.
		In view of the intertwining relations established above \eqref{eq:q} to \eqref{intertwining_phi }, the following operators are well defined, contractive, and multi-analytic
		\begin{align}
			\Theta_{\mathcal L} &\coloneqq \Phi^{\mathcal{L}_*} Q (\Phi^{\mathcal{L}})^*,\label{Theta_L2} \\
			\Theta_1 &\coloneqq \Phi^{\mathcal{F}_1} Q_1 (\Phi^{\mathcal{L}})^*, \\
			\Theta_{i+1} &\coloneqq \Phi^{\mathcal{F}_{i+1}} Q_{i+1} (\Phi^{\mathcal{F}_i})^*, \\
			\Theta_k &\coloneqq \Phi^{\mathcal{L}_*} Q_k (\Phi^{\mathcal{F}_{k-1}})^*.
		\end{align}
		Applying the Fourier representations $\Phi^{\mathcal L_*}$ to the projection identity \eqref{proj:L_*}, we obtain 
		\begin{align*}
			\Phi^{\mathcal{L}_*} P^{\mathcal{L}_*} l &= \Phi^{\mathcal{L}_*} P^{\mathcal{L}_*} (\Phi^{\mathcal{F}_{k-1}})^* \Phi^{\mathcal{F}_{k-1}} P^{\mathcal{F}_{k-1}} \dots (\Phi^{\mathcal{F}_1})^* \Phi^{\mathcal{F}_1} P^{\mathcal{F}_1} l \\
			\Theta_{\mathcal L} \Phi^{\mathcal L}l&= \Theta_k \Theta_{k-1} \dots \Theta_1 \Phi^\mathcal{L} l.
		\end{align*}
		Hence we have the following factorization 
		\begin{equation}\label{theta_l_factorization2}
			\Theta_{\mathcal L} = \Theta_k \Theta_{k-1} \dots \Theta_1  .
		\end{equation}
		Utilizing  relation given in \eqref{P_RM(L)}, we define the operator $\Phi_{\mathcal{R}} : \mathcal{R} \to \overline{\Delta_{\mathcal{L}} (\Gamma \otimes \mathcal{L})}$ by
		\begin{equation*}
			\Phi_{\mathcal{R}}(P_{\mathcal{R}} l) \coloneqq \Delta_{\mathcal{L}} \Phi^{\mathcal{L}} l,
		\end{equation*}
		where the defect operator  $\Delta_{\mathcal{L}} = (I - \Theta^*_{\mathcal L} \Theta_{\mathcal L})^{1/2}$. To prove that $\Phi_{\mathcal{R}}$ is  unitary, it suffices to show that it is an isometry on a dense subset. Let $l \in M_V(\mathcal{L})$. Invoking \eqref{eq:PR}, we compute as follows
		\begin{align*}
			\|P_{\mathcal{R}} l\|^2 &= \|l\|^2 - \|P^{\mathcal{L}_*} l\|^2 \nonumber \\
			&= \|\Phi^{\mathcal{L}} l\|^2 - \|\Phi^{\mathcal{L}_*} P^{\mathcal{L}_*} l\|^2 \nonumber \\
			&= \|\Phi^{\mathcal{L}} l\|^2 - \|\Theta_{\mathcal L} \Phi^{\mathcal{L}} l\|^2 \nonumber \\
			&= \|\Delta_{\mathcal{L}} \Phi^{\mathcal{L}} l\|^2.
		\end{align*}
	Analogous arguments yield unitary maps associated with the spaces $\mathcal{W}_i$
		\begin{align*}
			\Phi_{\mathcal{W}_1} &: \mathcal{W}_1 \to \overline{\Delta_1 (\Gamma \otimes \mathcal{L})}, ~ \Phi_{\mathcal{W}_1}(P_{\mathcal{W}_1} l) = \Delta_1 \Phi^{\mathcal{L}} l, \\
			\Phi_{\mathcal{W}_{i+1}} &: \mathcal{W}_{i+1} \to \overline{\Delta_{i+1} (\Gamma \otimes \mathcal{F}_i)}, ~ \Phi_{\mathcal{W}_{i+1}}(P_{\mathcal{W}_{i+1}} f_i) = \Delta_{i+1} \Phi^{\mathcal{F}_i} f_i ~ \text{for } i = 1, \dots, k-2, \\
			\Phi_{\mathcal{W}_k} &: \mathcal{W}_k \to \overline{\Delta_k (\Gamma \otimes \mathcal{F}_{k-1})}, ~ \Phi_{\mathcal{W}_k}(P_{\mathcal{W}_k} f_{k-1}) = \Delta_k \Phi^{\mathcal{F}_{k-1}} f_{k-1},
		\end{align*}
		where $\Delta_i = (I - \Theta_i^* \Theta_i)^{1/2}$ for $i = 1, \dots, k$.
		Since the maximal reducing subspace admits the orthogonal decomposition $\mathcal{R} = \mathcal{W}_k \oplus \cdots \oplus \mathcal{W}_1$, we define the operator $Z_k$ by
		\begin{equation*}
			Z_k \coloneqq \left( \Phi_{\mathcal{W}_k} \oplus \cdots \oplus \Phi_{\mathcal{W}_1} \right) \Phi_{\mathcal{R}}^{-1}.
		\end{equation*}
		By construction, $Z_k$ is a unitary map from $\overline{\Delta_{\mathcal{L}}(\Gamma \otimes \mathcal{L})}$ onto the orthogonal direct sum $\overline{\Delta_k(\Gamma \otimes \mathcal{F}_{k-1})} \oplus \cdots \oplus \overline{\Delta_1(\Gamma \otimes \mathcal{L})}$. 
		We compute the explicit action of $Z_k$ on the dense set of elements of the form $\Delta_{\mathcal{L}} \Phi^{\mathcal{L}} l$, using the expresion of $P_{\mathcal R}l$  given in \eqref{proj:R} for $l\in M_v(\mathcal L)$
		\begin{align*}
			Z_k(\Delta_{\mathcal{L}} \Phi^{\mathcal{L}} l) &= \left( \Phi_{\mathcal{W}_k} \oplus \cdots \oplus \Phi_{\mathcal{W}_1} \right) P_{\mathcal{R}} l \nonumber \\
			&= \Phi_{\mathcal{W}_k} P_{\mathcal{W}_k} P^{\mathcal{F}_{k-1}} \cdots P^{\mathcal{F}_1} l \oplus \cdots \oplus \Phi_{\mathcal{W}_2} P_{\mathcal{W}_2} P^{\mathcal{F}_1} l \oplus \Phi_{\mathcal{W}_1} P_{\mathcal{W}_1} l \nonumber \\
			&= \Delta_k \Theta_{k-1} \cdots \Theta_1 \Phi^{\mathcal{L}} l \oplus \cdots \oplus \Delta_2 \Theta_1 \Phi^{\mathcal{L}} l \oplus \Delta_1 \Phi^{\mathcal{L}} l.
		\end{align*}
		For $v\in \Gamma \otimes \mathcal{L} $, this yields
		\begin{equation*}
			Z_k \Delta_{\mathcal{L}} v = \Delta_k \Theta_{k-1} \cdots \Theta_1 v \oplus \cdots \oplus \Delta_2 \Theta_1 v \oplus \Delta_1 v.
		\end{equation*}
		Since $Z_k$ is a unitary operator, we immediately obtain
		\begin{equation*}
			\overline{Z_k(\Delta_{\mathcal{L}}(\Gamma \otimes \mathcal{L}))} = \overline{\Delta_k(\Gamma \otimes \mathcal{F}_{k-1})} \oplus \cdots \oplus \overline{\Delta_2(\Gamma \otimes \mathcal{F}_1)} \oplus \overline{\Delta_1(\Gamma \otimes \mathcal{L})}.
		\end{equation*}
		Observe that the Fourier representation map $\Phi$, defined by
		\begin{equation*}
			\Phi \coloneqq \Phi^{\mathcal{L}_*} \oplus \Phi_{\mathcal{R}} : \mathcal{K} \longrightarrow \wh{\mathcal{K}} \coloneqq (\Gamma \otimes \mathcal{L}_*) \oplus \overline{\Delta_{\mathcal{L}}(\Gamma \otimes \mathcal{L})},
		\end{equation*}
		is a unitary operator from the minimal isometric dilation space onto the functional model space. Consequently, by substituting our expression for the residual part, we can rewrite $\Phi$ as
		\begin{equation}\label{Phi_map}
			\Phi = \Phi^{\mathcal{L}_*} \oplus Z_k^{*} \left( \Phi_{\mathcal{W}_k} \oplus \cdots \oplus \Phi_{\mathcal{W}_1} \right).
		\end{equation}	
		Finally, we determine the image of the Hilbert space $\mathcal{H}$, which sits inside $\mathcal{K}$ as the orthogonal complement $\mathcal{H} = \mathcal{K} \ominus M_V(\mathcal{L})$. Under the unitary operator $\Phi$, the functional model for $\mathcal{H}$ takes the precise form
		\begin{equation}\label{model_space_H}
			\Phi(\mathcal{H}) = \wh{\mathcal{H}} = \left[ (\Gamma \otimes \mathcal{L}_*) \oplus \overline{\Delta_{\mathcal{L}}(\Gamma \otimes \mathcal{L})} \right] \ominus \left\{ \Theta u \oplus \Delta_{\mathcal{L}} u : u \in \Gamma \otimes \mathcal{L} \right\}.
		\end{equation}
		Under the unitary transformation $\Phi$, the row contraction $T = [T_1, \ldots, T_n]$ is carried to the row contraction $\wh{T} = [\wh{T}_1, \ldots, \wh{T}_n]$. For each $j = 1, \ldots, n$, the adjoint operator $\wh{T}_j^*$ acts according to
		\begin{equation*}
			\wh{T}_j^* (l_* \oplus \Delta_{\mathcal L} l) 
			= (S_j^* \otimes I_{\mathcal L_*}) l_* \oplus C_j^* (\Delta_{\mathcal L} l), 
			\quad l \in \gt \mathcal L, \; l_* \in \gt \mathcal L_*,
		\end{equation*}
		where the operator 
		\[
		C_j : \ov{\Delta_{\mathcal L}(\gt \mathcal L)} \longrightarrow \ov{\Delta_{\mathcal L}(\gt \mathcal L)}
		\]
		is defined by
		\begin{equation*}
			C_j (\Delta_{\mathcal L} l) 
			\coloneqq \Delta_{\mathcal L} (S_j \otimes I_{\mathcal L}) l, 
			\quad l \in \gt \mathcal L.
		\end{equation*}
		
	%%%%%%%%%%%%%%%%%%%%%%%%%%%%%%%%%%%%%%%%%%%%%%%%%%%%%%%%	
		The images of the joint invariant subspaces $\mathcal{M}_i$ under the unitary map $\Phi$ are now determined. Recall from \eqref{M_i} that
		\begin{equation}
			\mathcal{M}_i = (M_V(\mathcal{F}_i) \oplus \mathcal{R}_i) \ominus M_V(\mathcal{L}) \quad \text{for } i=1, \dots, k-1.
		\end{equation}
		Applying the Fourier representation $\Phi$, which preserves orthogonal direct sums and orthogonal complements, yields
		\begin{equation}\label{eq:phi_Mi_decomp}
			\Phi(\mathcal M_i) = \Phi(M_V(\mathcal{F}_i) \oplus \mathcal R_i) \ominus \Phi(M_V(\mathcal{L})).
		\end{equation}
		To determine $\Phi(M_V(\mathcal{F}_i))$, consider an arbitrary vector
		$f_i \in M_V(\mathcal{F}_i)$.
	Iterative application of the relations
		\eqref{L}--\eqref{f_k-1}
		gives
		\begin{align*}
			\Phi(M_V(\mathcal F_i)) 
			&= \Phi \Big\{ 
			P^{\mathcal{L}_*} P^{\mathcal F_{k-1}} \cdots P^{\mathcal F_{i+1}} f_i
			\oplus 
			P_{\mathcal W_k} P^{\mathcal F_{k-1}} \cdots P^{\mathcal F_{i+1}} f_i  \\
			&\qquad\qquad 
			\oplus \cdots 
			\oplus 
			P_{\mathcal W_{i+1}} f_i 
			:\,
			f_i \in M_V(\mathcal F_i) 
			\Big\}  \\
			&= \Big\{ 
			\Theta_k \Theta_{k-1} \cdots \Theta_{i+1} u_i  \\
			&\qquad \oplus 
			Z_k^{*}\Big(
			\Delta_k \Theta_{k-1} \cdots \Theta_{i+1} u_i
			\oplus \cdots 
			\oplus 
			\Delta_{i+1} u_i  \\
			&\qquad\qquad
			\oplus 
			\underbrace{0 \oplus \cdots \oplus 0}_{i\text{ terms}}
			\Big)
			:\,
			u_i \in \Gamma \otimes \mathcal F_i
			\Big\}.
		\end{align*}
		where $u_i = \Phi^{\mathcal{F}_i} f_i \in \Gamma \otimes \mathcal{F}_i$.
		Next, consider the representation of the residual space $\mathcal R_i$. Since, 
		$$\mathcal R_i = \mathcal W_i \oplus \cdots \oplus \mathcal W_1,$$
		 it follows that
		\begin{align*}
			\Phi(\mathcal R_i) &= \Phi(\mathcal W_i \oplus \cdots \oplus \mathcal W_1) \\
			&= Z_k^{*} \left( \underbrace{\{0\} \oplus \cdots \oplus \{0\}}_{k-i \text{ times}} \oplus \ov{\Delta_i (\Gamma \otimes \mathcal{F}_{i-1})} \oplus \cdots \oplus \ov{\Delta_2 (\Gamma \otimes \mathcal{F}_1)} \oplus \ov{\Delta_1 (\Gamma \otimes \mathcal{L})} \right).
		\end{align*}
		
		Combining these two representation, we get
		\begin{multline*}
			\Phi(M_V(\mathcal{F}_i) \oplus \mathcal R_i) = \left\{ \Theta_k \cdots \Theta_{i+1} u_i \oplus Z_k^{*} (\Delta_k \Theta_{k-1} \cdots \Theta_{i+1} u_i \oplus \cdots \oplus \Delta_{i+1} u_i \oplus v_i \oplus \cdots \right. \\ \oplus v_1)  : 
			\left. u_i \in \Gamma \otimes \mathcal{F}_i, \ v_j \in \ov{\im(\Delta_j)} \text{ for } j=1,\dots,i \right\}.
		\end{multline*}
		Finally, substituting this combined explicit form back into our orthogonal decomposition \eqref{eq:phi_Mi_decomp} alongside the known representation $\Phi(M_V(\mathcal{L})) = \{ \Theta u \oplus \Delta_{\mathcal{L}} u : u \in \Gamma \otimes \mathcal{L} \}$, we arrive at the complete expression for $\wh{\mathcal{M}}_i \coloneqq \Phi(\mathcal M_i)$
		\begin{multline}\label{eq:phi_Mi}
			\wh{\mathcal{M}}_i  = \left\{ \Theta_k \cdots \Theta_{i+1} u_i \oplus Z_k^{*} (\Delta_k \Theta_{k-1} \cdots \Theta_{i+1} u_i \oplus \cdots \oplus \Delta_{i+1} u_i \oplus v_i \oplus \cdots \oplus v_1) : \right. \\
			\left. u_i \in \Gamma \otimes \mathcal{F}_i, \ v_j \in \ov{\im( \Delta_j)} \text{ for } j=1,\dots,i \right\} \ominus \left\{ \Theta u \oplus \Delta_{\mathcal{L}} u : u \in \Gamma \otimes \mathcal{L} \right\}.
		\end{multline}
		Moreover, the orthogonal complement $\wh{\mathcal N}_i$ is computed as follows
		\begin{align*}
			\wh{\mathcal{N}}_i &=\wh{\mathcal H} \ominus \wh{\mathcal M}_i\\
			&= \left[ \left( \Gamma \otimes \mathcal{L}_* \oplus \ov{\Delta_{\mathcal{L}} (\Gamma \otimes \mathcal{L})} \right) \ominus \{ \Theta u \oplus \Delta_{\mathcal{L}} u : u \in \Gamma \otimes \mathcal{L} \} \right] \\
			&\quad \ominus \Bigg[ \Big\{ \Theta_k \cdots \Theta_{i+1} u_i \oplus Z_k^{*}(\Delta_k \cdots \Theta_{i+1} u_i \oplus \cdots \oplus \Delta_{i+1} u_i \oplus v_i \oplus \cdots \oplus v_1) \\
			&\qquad \qquad : u_i \in \Gamma \otimes \mathcal{F}_i, \ v_j \in \ov{\im(\Delta_j)} \text{ for } j=1,\dots,i \Big\} \ominus \{ \Theta u \oplus \Delta_{\mathcal{L}} u : u \in \Gamma \otimes \mathcal{L} \} \Bigg] \\
			&= \left[ (\Gamma \otimes \mathcal{L}_*) \oplus Z_k^{*} \left( \ov{\Delta_k (\Gamma \otimes \mathcal{F}_{k-1})} \oplus \cdots \oplus \ov{\Delta_{i+1} (\Gamma \otimes \mathcal{F}_i)} \oplus \underbrace{\{0\} \oplus \cdots \oplus \{0\}}_{i \text{ times}} \right) \right] \\
			&\quad \ominus \left\{ \Theta_k \cdots \Theta_{i+1} u_i \oplus Z_k^{*}(\Delta_k \cdots \Theta_{i+1} u_i \oplus \cdots \oplus \Delta_{i+1} u_i \oplus \underbrace{0 \oplus \cdots \oplus 0}_{i \text{ times}}) : u_i \in \Gamma \otimes \mathcal{F}_i \right\}.
		\end{align*}
		Recall from  \cite{Po89a} that G. Popescu proved that $\Theta_{\mathcal L}$ (defined in \ref{Theta_L2}) coincides with the characteristic function  $\Theta_T$ of the row contraction $T$. Hence, $\wt{\Theta}$ coincides with $\Theta_{\mathcal L}$. Consequently, the $k$-regular factorization given by \eqref{theta_l_factorization2} yields the corresponding $k$-regular factorization
		\begin{equation}
			\wt {\Theta} = \wt {\Theta}_k\cdots \wt {\Theta}_1,
		\end{equation}
		where each \( \wt {\Theta}_i(z):\Gamma\otimes \mathcal{E}_i\to \Gamma\otimes \mathcal{E}_{i+1} \) is a contractive multi-analytic operator for \( i=1,\dots,k \), with initial space \( \mathcal{E}_1 = \mathcal{E} \) and final space \( \mathcal{E}_{k+1} = \mathcal{E}_* \). Moreover, the functional model pair \( (\wh{T}, \wh{\mathcal H}) \) corresponds to the functional model pair \( (\wt{T}, \wt{\mathcal H}) \), and for \( i=1,\dots,k-1 \), the subspaces \( \wh{\mathcal M}_i \) and \( \wh{\mathcal N}_i \) correspond to \( \wt{\mathcal M}_i \) and \( \wt{\mathcal N}_i \), respectively.
		
		%%%%%%%%%%%%%%%%%%%%%%%%%%%%%%%%%%%%%%%%%%%%%%%%%%%

		Conversely, let $\wt{\Theta} = \wt{\Theta}_k \cdots \wt{\Theta}_1$ be a $k$-regular factorization, where each \( \wt{\Theta}_i:\Gamma\otimes\mathcal{E}_i\to \Gamma\otimes \mathcal{E}_{i+1}  \) is a contractive multi- analytic operator for \( i=1,\dots,k \), with initial space \( \mathcal{E}_1 = \mathcal{E} \) and final space \( \mathcal{E}_{k+1} = \mathcal{E}_* \). For each \( i = 1, \dots, k-1 \), let the subspaces \( \wt{\mathcal M}_i \) and \( \wt{\mathcal N} _i\) be defined by
		\begin{align*}
			\wt{\mathcal{M}}_{i} 
			&= \Big\{ 
			{\wt{\Theta}}_k \cdots {\wt{\Theta}}_{i+1} u_{i+1} 
			\oplus 
			{Z_k}^{*}\big(
			\wt  \Delta_k {\wt{\Theta}}_{k-1} \cdots {\wt{\Theta}}_{i+1} u_{i+1} 
			\oplus \cdots 
			\oplus \wt  \Delta_{i+1} u_{i+1} 
			\oplus v_i 
			\oplus \cdots 
			\oplus v_1
			\big)   \\
			&\qquad 
			: u_{i+1} \in \Gamma\otimes\mathcal{E}_{i+1}, \; 
			v_j \in \overline{\wt  \Delta_j (\Gamma\otimes(\mathcal{E}_j)}, \; j=1,\dots,i
			\Big\} 
			\ominus \wt{\mathcal G},\\
			\wt { \mathcal N}_i&= \left[ \Gamma\otimes\mathcal{E}_* \oplus {Z_k}^{*} \left( \ov{\wt  \Delta_k (\Gamma\otimes(\mathcal{E}_{k})} \oplus \cdots \oplus \ov{\wt  \Delta_{i+1} (\Gamma\otimes(\mathcal{E}_{i+1})} \oplus \{0\} \oplus \cdots \oplus \{0\} \right) \right] \\
			&\quad \ominus \left\{  {\wt{\Theta}}_k \cdots  {\wt{\Theta}}_{i+1} u_{i+1} + {Z_k}^{*} ( \wt  \Delta_k  {\wt{\Theta}}_{k-1} \cdots {\wt{\Theta}}_{i+1} u_{i+1} \oplus \cdots \oplus \wt  \Delta_{i+1} u_{i+1}  \oplus 0 \oplus \cdots \oplus 0 )  \right. \\
			&  \qquad \left.: u_{i+1} \in \Gamma\otimes\mathcal{E}_{i+1} \right\}. \\
		\end{align*}

		Since $\wt{Z}_k$ is unitary, we define the subspace $\wt{\mathcal G}_i$ by
		\begin{align*}
			\wt{\mathcal G}_i &\coloneqq \left\{ \wt{\Theta}_k \cdots \wt{\Theta}_{i+1} u_{i+1} + \wt{Z}_k^{*} ( \wt  \Delta_k \wt{\Theta}_{k-1} \cdots \wt{\Theta}_{i+1} u_{i+1} \oplus \cdots \oplus \wt  \Delta_{i+1} u_{i+1} \oplus v_i \oplus \cdots \oplus v_1 ) \right. \\
			&\qquad \left. : u_{i+1} \in \Gamma\otimes\mathcal{E}_{i+1}, \ v_j \in \ov{\im(\wt  \Delta_j)}, \ j=1,\dots,i \right\} \\
			&\supseteq \left\{ \wt{\Theta}_k \cdots \wt{\Theta}_1 u \oplus \wt{Z}_k^{*} ( \wt  \Delta_k \wt{\Theta}_{k-1} \cdots \wt{\Theta}_1 u \oplus \cdots \oplus \wt  \Delta_{i+1} \wt{\Theta}_i \cdots \wt{\Theta}_1 u \oplus \cdots \oplus \wt  \Delta_1 u ) \right. \\
			&\qquad \left. : u \in \Gamma\otimes\mathcal{E} \right\} \\
			&= \{ \wt{\Theta} u \oplus \wt  \Delta u : u \in \Gamma\otimes\mathcal{E}\} = \wt{\mathcal G}.
		\end{align*}
		It follows that $ \wt{\mathcal G}_i \supseteq \wt{\mathcal G},$ and we have $ \wt{\mathcal M}_i = \wt{\mathcal G}_i \ominus \wt{\mathcal G}.$
		Observe that 
		\begin{align*}
			&(\Gamma\otimes\mathcal{E}_* \oplus \ov{\wt  \Delta (\Gamma\otimes(\mathcal{E})}) \ominus \wt{\mathcal G}_i \\
			&= \left[ \Gamma\otimes\mathcal{E}_* \oplus \ov{\wt  \Delta ((\Gamma\otimes(\mathcal{E}))} \right] \ominus \left\{ \wt{\Theta}_k \cdots \wt{\Theta}_{i+1} u_{i+1} + \wt{Z}_k^{*} ( \wt  \Delta_k \wt{\Theta}_{k-1} \cdots \wt{\Theta}_{i+1} u_{i+1} \oplus \cdots \right. \\
			&\qquad \left. \cdots \oplus \wt  \Delta_{i+1} u_{i+1} \oplus v_i \oplus \cdots \oplus v_1 ) : u_{i+1} \in \Gamma\otimes\mathcal{E}_{i+1}, \ v_j \in \ov{\im(\wt  \Delta_j)}, \ j=1,\dots,i \right\} \\
			&= \left[ \Gamma\otimes\mathcal{E}_* \oplus \wt{Z}_k^{*} \left( \ov{\wt  \Delta_k (\Gamma\otimes(\mathcal{E}_{k})} \oplus \cdots \oplus \ov{\wt  \Delta_{i+1} (\Gamma\otimes(\mathcal{E}_{i+1})} \oplus \{0\} \oplus \cdots \oplus \{0\} \right) \right] \\
			&\quad \ominus \left\{ \wt{\Theta}_k \cdots \wt{\Theta}_{i+1} u_{i+1} + \wt{Z}_k^{*} ( \wt  \Delta_k \wt{\Theta}_{k-1} \cdots \wt{\Theta}_{i+1} u_{i+1} \oplus \cdots \oplus \wt  \Delta_{i+1} u_{i+1} \oplus 0 \oplus \cdots \oplus 0 ) \right. \\
			&\qquad \left. : u_{i+1} \in \Gamma\otimes\mathcal{E}_{i+1}\right\} \\
			&= \wt{\mathcal N}_i.
		\end{align*}
		Hence, $$\wt{\mathcal M}_i \oplus \wt{\mathcal N}_i = [\Gamma\otimes\mathcal{E}_* \oplus \ov{\wt  \Delta (\Gamma\otimes\mathcal{E})}]\ominus\wt{\mathcal G}=\wt{\mathcal H}.$$
		
		It remains to prove that the subspaces
		$\wt{\mathcal M}_1, \dots, \wt{\mathcal M}_{k-1}$
	are joint invariant subspaces for the row contraction $\wt{T} = [\wt{T}_1, \dots, \wt{T}_n]$. To establish this, it is sufficient to prove that $\wt{T}_j^*(\wt{\mathcal N}_i) \sub \wt{\mathcal N}_i$ for each operator $j = 1, \dots, n$ and $i = 1, \dots, k-1.$
		
		If $x \in \wt{\mathcal N}_i$, there exist elements $u \in \Gamma \otimes \mathcal{E}_*$ and $v_m \in \ov{\im \wt  \Delta_m}$ (for $m = i+1, \dots, k$) such that
		\begin{equation}
			x = u \oplus Z_k^{*}(v_k \oplus \dots \oplus v_{i+1} \oplus \underbrace{0 \oplus \cdots \oplus 0}_{i \text{ times}}),
		\end{equation}
		which satisfies the following orthogonality relation for every $u_{i+1}\in \Gamma \otimes \mathcal{E}_{i+1}$,
		\begin{multline}\label{orthogonal_M}
			\Big\langle u \oplus Z_k^{*}(v_k \oplus \dots \oplus v_{i+1} \oplus 0 \oplus \cdots \oplus 0), \\
			\Theta_k \cdots\wt  \Theta_{i+1} u_{i+1}\oplus Z_k^{*}(\wt  \Delta_k\wt  \Theta_{k-1} \cdots\wt  \Theta_{i+1} u_{i+1}\oplus \dots \oplus \wt  \Delta_{i+1} u_{i+1}\oplus 0 \oplus \cdots \oplus 0) \Big\rangle = 0.\\
			\langle\wt  \Theta_{i+1}^* \cdots\wt  \Theta_k^* u +\wt  \Theta_{i+1}^* \cdots\wt  \Theta_{k-1}^* \wt  \Delta_k v_k + \dots + \wt  \Delta_{i+1} v_{i+1}, \ u_{i+1}\rangle = 0.
		\end{multline}
		Since this equality hold for all $u_{i+1}\in \Gamma \otimes \mathcal{E}_{i+1}$, we have
		\begin{equation}\label{eq:orthogonality_general}
			\Theta_{i+1}^* \cdots\wt  \Theta_k^* u +\wt  \Theta_{i+1}^* \cdots\wt  \Theta_{k-1}^* \wt  \Delta_k v_k + \dots + \wt  \Delta_{i+1} v_{i+1} =0.
		\end{equation}
		
		Now, applying $\wt{T}_j^*$ to $x$ for an arbitrary $1 \le j \le n$, we obtain
		\begin{align*}
			\wt{T}_j^*(x) &= (S_j^* \otimes I_{\mathcal{L}_*}) u \oplus C_j^* Z_k^{*}(v_k \oplus \dots \oplus v_{i+1} \oplus 0 \oplus \cdots \oplus 0) \\
		\end{align*}

		%%%%%%%%%%%%%%%%%%%%%
		To establish that $\wt{T}_j^*(x) \in \wt{\mathcal N}_i$, it suffices to verify that $\wt{T}_j^*(x)$ satisfies the orthogonality condition \eqref{orthogonal_M}. For a fixed $u_{i+1}\in \Gamma \otimes \mathcal{E}_{i+1}$, we define the vector $y$ by     
		\begin{equation*}
			y =\wt  \Theta_k \cdots\wt  \Theta_{i+1} u_{i+1}\oplus Z_k^{*}(\wt  \Delta_k\wt  \Theta_{k-1} \cdots\wt  \Theta_{i+1} u_{i+1}\oplus \dots \oplus \wt  \Delta_{i+1} u_{i+1}\oplus 0 \oplus \cdots \oplus 0).
		\end{equation*}
		Exploiting the fact that the operators $\Theta_i$ are multi-analytic and that 
		\[
		Z_k C_j^* Z_k^{*} = \operatorname{diag}\{{C_j^{(k)}}^*, \dots, {C_j^{(1)}}^*\},
		\]
		we compute as follows
		\begin{align*}
			\langle \wt{T}_j^*(x), y \rangle &= \Big\langle (S_j^* \otimes I_{\mathcal{L}_*}) u \oplus C_j^* Z_k^{*}(v_k \oplus \dots \oplus v_{i+1} \oplus 0 \oplus \cdots \oplus 0),  \\
			&\qquad\wt  \Theta_k \cdots\wt  \Theta_{i+1} u_{i+1}\oplus Z_k^{*}(\wt  \Delta_k\wt  \Theta_{k-1} \cdots\wt  \Theta_{i+1} u_{i+1}\oplus \dots \oplus \wt  \Delta_{i+1} u_{i+1}\oplus 0 \oplus \cdots \oplus 0) \Big\rangle \\
			&= \langle (S_j^* \otimes I_{\mathcal{L}_*})\wt  \Theta_{i+1}^* \cdots\wt  \Theta_k^* u, u_{i+1}\rangle + \Big\langle Z_k C_j^* Z_k^{*} (v_k \oplus \dots \oplus v_{i+1} \oplus 0 \oplus \cdots \oplus 0), \\
			&\qquad \wt  \Delta_k\wt  \Theta_{k-1} \cdots\wt  \Theta_{i+1} u_{i+1}\oplus \dots \oplus \wt  \Delta_{i+1} u_{i+1}\oplus 0 \oplus \cdots \oplus 0 \Big\rangle \\
			&= \langle (S_j^* \otimes I_{\mathcal{L}_*})\wt  \Theta_{i+1}^* \cdots\wt  \Theta_k^* u, u_{i+1}\rangle \\
			&\quad + \langle C_j^{k*} v_k \oplus \cdots \oplus C_j^{(i+1)*} v_{i+1}, \wt  \Delta_k\wt  \Theta_{k-1} \cdots\wt  \Theta_{i+1} u_{i+1}\oplus \dots \oplus \wt  \Delta_{i+1} u_{i+1}\rangle \\
			&= \Big\langle (S_j^* \otimes I_{\mathcal{L}_*})\wt  \Theta_{i+1}^* \cdots\wt  \Theta_k^* u 
			+\wt  \Theta_{i+1}^* \cdots\wt  \Theta_{k-1}^* \wt  \Delta_k C_j^{k*} v_k 
			+ \dots 
			+ \wt  \Delta_{i+1} C_j^{(i+1)*} v_{i+1}, 
			\, u_{i+1}\Big\rangle.
		\end{align*}
		Applying the intertwining property $\wt  \Delta_m C_j^{m*} = (S_j^* \otimes I) \wt  \Delta_m$, we can reformulate the inner product as follows
		\begin{align*}
			\langle \wt{T}_j^*(x), y \rangle &= \langle (S_j^* \otimes I)\wt  \Theta_{i+1}^* \cdots\wt  \Theta_k^* u +\wt  \Theta_{i+1}^* \cdots\wt  \Theta_{k-1}^* (S_j^* \otimes I) \wt  \Delta_k v_k + \dots + (S_j^* \otimes I) \wt  \Delta_{i+1} v_{i+1}, u_{i+1}\rangle \\
			&= \Big\langle (S_j^* \otimes I) \big(\wt  \Theta_{i+1}^* \cdots\wt  \Theta_k^* u +\wt  \Theta_{i+1}^* \cdots\wt  \Theta_{k-1}^* \wt  \Delta_k v_k + \dots + \wt  \Delta_{i+1} v_{i+1} \big), u_{i+1}\Big\rangle.
		\end{align*}
		Substituting the orthogonality condition established in \eqref{eq:orthogonality_general}, we deduce that
		\begin{equation*}
			\langle \wt{T}_j^*(x), y \rangle = \langle (S_j^* \otimes I) (0), u_{i+1}\rangle = 0.
		\end{equation*}
		Therefore, we conclude that $\wt{T}_j^*(x) \in \wt{\mathcal N}_i$ for all $j=1,\ldots,n$.
		
		Finally, we demonstrate the nesting $\wt{\mathcal M}_1 \sub \wt{\mathcal M}_2 \sub \dots \sub \wt{\mathcal M}_{k-1}$. For any  $i \in \{1, \dots, k-2\}$, recall that
		\begin{multline*}
			\wt{\mathcal M}_{i+1} = \left\{ \wt{\Theta}_k \cdots \wt{\Theta}_{i+2} u_{i+2} \oplus \wt{Z}_k^{*} (\wt \Delta_k \wt{\Theta}_{k-1} \cdots \wt{\Theta}_{i+2} u_{i+2} \oplus \dots \oplus \wt \Delta_{i+2} u_{i+2} \oplus v_{i+1} \oplus \dots \oplus v_1) : \right. \\
			\left. u_{i+2} \in \Gamma\otimes\mathcal{E}_{i+2}, \ v_j \in \ov{\im(\wt \Delta_j)} \right\} \ominus \left\{ \wt{\Theta} u \oplus \wt \Delta u : u \in \Gamma\otimes\mathcal{E} \right\}.
		\end{multline*}
		By choosing $u_{i+2} = \wt{\Theta}_{i+1} u_{i+1}$, it is clear that the elements generating $\wt{\mathcal M}_i$ are contained within $\wt{\mathcal M}_{i+1}$. Hence, the inclusion $\wt{\mathcal M}_i \sub \wt{\mathcal M}_{i+1}$ follows directly.
	\end{proof}
	%%%%%%%%%%%%%%%%%%%%%%%%%%%%%%%%%%%%%%%%%%%%%%%%%%%%%%%%%%%%%%%%%%%%%%%%%%%%%%%%%%%%%%%%%%%%%%%%%%%%%%%%%%%%%%%%%%%%%%%%%%%
	The following theorem constructs a functional model associated with a given $k$-regular factorization of a contractive multi-analytic operator  that satisfies the Szeg\H{o} condition (see condition \ref{szego_condition}, Lemma \ref{cuntz_isometry}). Moreover, it characterizes the joint invariant subspaces of the corresponding model operator induced by this factorization.
	\begin{theorem}\label{main_thereom_2}
		Let \( T \) be a c.n.c. row contraction on a Hilbert space \( \mathcal H \), and let \( \Theta : \Gamma \otimes \mathcal E \to \Gamma \otimes \mathcal E_* \) be a contractive multi-analytic operator that coincides with the characteristic function of \( T \).	Assume that for an integer \( k \geq 2 \), \( \Theta \) admits a \( k \)-regular factorization of the form
		\begin{equation}\label{k_regular_fact21}
			\Theta = \Theta_k\cdots \Theta_1,
		\end{equation}
		where, for each \( i = 1, \dots, k \), the factor \( \Theta_i: \Gamma\otimes\mathcal{E}_i\to \Gamma\otimes\mathcal{E}_{i+1}\) is a contractive multi-analytic operator, with spaces \( \mathcal{E}_1 = \mathcal{E} \) and \( \mathcal{E}_{k+1} = \mathcal{E}_* \).
		
		Then, the operator \( T \) is unitarily equivalent to the row operator \( \mathbf{T} = [\mathbf{T}_1, \dots, \mathbf{T}_n] \) acting on the Hilbert space 
		\[
		\boldsymbol{\mathcal H}
		\coloneqq
		\Bigl[
		(\Gamma \otimes \mathcal{E}_{k+1})
		\oplus
		\overline{\Delta_k (\Gamma \otimes \mathcal{E}_k)}
		\oplus \cdots \oplus
		\overline{\Delta_1 (\Gamma \otimes \mathcal{E}_1)}
		\Bigr]
		\ominus \boldsymbol {\mathcal G}.
		\]
		For each \( m = 1, \dots, n \), the adjoint operators act via
		\[
		\boldsymbol{T}_m^*(f \oplus g_k \oplus \cdots \oplus g_1)
		=
		(S_m^* \otimes I_{\mathcal{E}_{k+1}}) f
		\oplus
		C_m^{(k)*} g_k
		\oplus \cdots \oplus
		C_m^{(1)*} g_1,
		\]
		where the subspace \( \boldsymbol {\mathcal G} \) and the defect operators \( \Delta_j \) are given by
		\[
		\boldsymbol {\mathcal G}
		=
		\left\{
		\Theta_k \cdots \Theta_1 f
		\oplus
		\Delta_k \Theta_{k-1} \cdots \Theta_1 f
		\oplus \cdots \oplus
		\Delta_1 f
		:
		f \in \Gamma \otimes \mathcal{E}_1
		\right\},
		\quad
		\Delta_j \coloneqq (I - \Theta_j^* \Theta_j)^{1/2}.
		\]
		
		Here, \( S_1, \dots, S_n \) denote the left creation operators on the free Fock space \( \Gamma \), and for each \( j = 1, \dots, k \), the tuple \( C^{(j)} = [C_1^{(j)}, \dots, C_n^{(j)}] \) represents the row isometry associated with \( \Theta_j \), as defined in Lemma $\ref{cuntz_isometry}$.
		
		Furthermore, the \( k \)-regular factorization \eqref{k_regular_fact21} induces a sequence of joint invariant subspaces explicitly described by
		\begin{align*}
			\boldsymbol{\mathcal{M}}_i
			=
			\Big\{
			&\Theta_k \cdots \Theta_{i+1} f_{i+1}
			\oplus
			\Delta_k \Theta_{k-1} \cdots \Theta_{i+1} f_{i+1}
			\oplus \cdots \oplus
			\Delta_{i+1} f_{i+1} \oplus g_i \oplus \cdots \oplus g_1 \\
			&
			:
			f_{i+1} \in \Gamma \otimes \mathcal{E}_{i+1},
			\;
			g_j \in \overline{\Delta_j(\Gamma \otimes \mathcal{E}_j)},
			\ j = 1, \dots, i
			\Big\}
			\ominus \mathcal G,
		\end{align*}
		for \( i = 1, \dots, k-1 \), which satisfy the ascending inclusion chain
		\[
		\boldsymbol{\mathcal{M}}_1 \sub \cdots \sub \boldsymbol{\mathcal{M}}_{k-1}.
		\]
	\end{theorem}
	\begin{proof}
		The result is an immediate consequence of the proof of Theorem~\ref{main_thereom_02} by replacing the unitary operator \( \Phi \), defined in~\eqref{Phi_map}, with the unitary operator
		\[
		\boldsymbol{\Phi} : \mathcal{K} \longrightarrow \boldsymbol{\mathcal{K}}
		\coloneqq
		\gt\mathcal{E}_*
		\oplus
		\overline{\Delta_k (\gt\mathcal{E}_k)}
		\oplus \cdots \oplus
		\overline{\Delta_1 (\gt\mathcal{E}_1)},
		\]
		given by
		\[
		\boldsymbol{\Phi}
		=
		\Phi^{\mathcal{L}_*}
		\oplus
		\bigl(
		\Phi_{\mathcal{W}_k}
		\oplus \cdots \oplus
		\Phi_{\mathcal{W}_1}
		\bigr).
		\]
	\end{proof}
	
	%%%%%%%%%%%%%%%%%%%%%%%%%%%%%%%%%%%%%%%%%%%%%%%%%%%%%%%%%%%%%%%%%%%%%%%%%%%%%%%%%%%%%%%%%%%%%%%%%%%%%%%%%%%%%%%
	\section {Upper Triangular Block Operator Matrix Form}
	For a given chain of joint invariant subspaces of a c.n.c.\ row contraction \(T\), the operator \(T\) admits an upper triangular block operator matrix representation. In this framework, the following result is obtained.
	\begin{theorem}\label{unitary_equivalent_in_multivarible}
		Let $\Theta : \gt\mathcal{E} \rightarrow \gt\mathcal{E}_{*}$ be a purely contractive multi-analytic operator satisfying
		\[
		\ov{\Delta_{\Theta}(\gt\mathcal{E})} = \ov{\Delta_{\Theta}[(\gt\mathcal{E})\ominus\mathcal{E}]}.
		\]
		Let $\Theta=\Theta_k\dots\Theta_1$ be a $k$-regular factorization of $\Theta$, where \( {\Theta}_i:\Gamma\otimes \mathcal{E}_i\to \Gamma\otimes \mathcal{E}_{i+1}  \) are contractive multi-analytic operators for \( i=1,\cdots,k \) with \( \mathcal{E}_1 = \mathcal{E} \) and \( \mathcal{E}_{k+1} = \mathcal{E}_* \). Let
		\[
		\mathcal M_1 \sub \mathcal M_2 \sub \dots \sub \mathcal M_{k-1}
		\]
		denote the corresponding chain of joint invariant subspaces for the model row contraction $T_{\Theta}=[T_{\Theta,1},\dots,T_{\Theta,n}]$ acting on the Hilbert space
		\[
		\mathcal{H}(\Theta) := [(\gt\mathcal{E}_{*}) \oplus \ov{\Delta_{\Theta}(\gt\mathcal{E})}] \ominus \mathcal{G}_{\Theta}.
		\]
		The action of the adjoint operator is given by
		\[
		T_{\Theta,j}^{*}(f \oplus g) := (S_j^* \otimes I_{\mathcal{E}_{*}})f \oplus C_j^*g, \quad \text{for all } j=1,\dots,n,
		\]
		where $S_1,\dots,S_n$ are the left creation operators on the full Fock space $\Gamma$, the operator $C=[C_1,\dots,C_n]$ is the row isometry as defined in Lemma \ref{cuntz_isometry}, and $\Delta_{\Theta}:=(I-\Theta^*\Theta)^{1/2}.$ Here, the subspace $\mathcal{G}_{\Theta}$ is defined as
		\(\mathcal{G}_{\Theta} = \{\Theta w \oplus \Delta_{\Theta}w : w \in \gt\mathcal{E}\}.\)
		Define the orthogonal difference spaces by
		\[
		\mathcal{H}_1 := \mathcal M_1, \quad \mathcal{H}_i := \mathcal M_i \ominus \mathcal M_{i-1} \text{ for } i=2,\dots,k-1,
		\]
		and
		\[
		\mathcal{H}_k := \mathcal{H}(\Theta) \ominus \mathcal M_{k-1}.
		\]
		With respect to the orthogonal decomposition
		\[
		\mathcal{H}(\Theta) = \mathcal{H}_1 \oplus \dots \oplus \mathcal{H}_k,
		\]
		each operator $T_{\Theta,j}$ admits the upper triangular block operator matrix representation
		\[
		T_{\Theta,j} = \begin{bmatrix} A_j^1 & \dots & * \\ 0 & \dots & * \\ \vdots & \ddots & \vdots \\ 0 & \dots & A_j^k \end{bmatrix},
		\]
		where $A_j^i = P_{\mathcal{H}_i}T_{\Theta,j}|_{\mathcal{H}_i}$.
		Then, for each $i=1,\dots,k$, the row contraction $A^i=[A_1^i,\dots,A_n^i]$ is unitarily equivalent to $T_{\Theta_i}=[T_{\Theta_i,1},\dots,T_{\Theta_i,n}]$, the standard functional model operator associated with the multi-analytic operator $\Theta_i$.
	\end{theorem}

	\begin{proof}
		We begin by establishing the necessary notation. For $i=1,\dots,k$, set
		\begin{equation*} \label{eq:2.1}
			\Phi_i \coloneq \Theta_k\dots\Theta_i
		\end{equation*}
		and
		\begin{equation*} \label{eq:2.2}
			\Lambda_i \coloneq [\Delta_k\Theta_{k-1}\dots\Theta_i, \dots, \Delta_{i+1}\Theta_i, \Delta_i]^T.
		\end{equation*}
		Then we have
		\begin{align} \label{eq:2.3}
			\Phi_i^*\Phi_i + \Lambda_i^*\Lambda_i &= \Theta_i^*\dots\Theta_k^*\Theta_k\dots\Theta_i + [\Theta_i^*\dots\Theta_{k-1}^*\Delta_k^2\Theta_{k-1}\dots\Theta_i + \dots + \Delta_i^2] \nonumber \\
			&= I_{\gt\mathcal{E}_i}.
		\end{align}
		Because the given factorization $\Theta=\Theta_k\dots\Theta_1$ is $k$-regular and $\Theta$ satisfies the Szeg\H{o} condition, from Lemma \ref{cuntz_isometry}, this implies that for $i=1,\dots,k$, the row isometries $C^{(i)}$ associated with $\Theta_i$ are Cuntz row isometries. Consequently, we have
		\[
		\ov{\Delta_{\Theta_i}(\gt\mathcal{E}_i)} = \ov{\Delta_{\Theta_i}[(\gt\mathcal{E}_i)\ominus\mathcal{E}_i]}, \quad \text{for all } i=1,\dots,k.
		\]
		For notational convenience,
		\(0_m\)
		denotes the zero vector in the \(m\)-fold direct sum of the corresponding orthogonal difference spaces.
		
		Recall that, for each $i=1,\dots,k$, the joint invariant subspace $\mathcal{M}_i$ admits the structural decomposition described in Theorem \ref{main_thereom_02}
		\begin{align*}
			\mathcal M_i &= \{ \Theta_k\dots\Theta_{i+1}f_{i+1} \oplus Z_k^*(\Delta_k\Theta_{k-1}\dots\Theta_{i+1}f_{i+1} \oplus \dots \oplus \Delta_{i+1}f_{i+1} \oplus g_i \oplus \dots \oplus g_1) \\
			&\quad : f_{i+1} \in \gt\mathcal{E}_{i+1}, \, g_j \in \ov{\im(\Delta_j)}, \, j=1,\dots,i \} \ominus \mathcal{G}_{\Theta}.
		\end{align*}
		\vspace{0.3cm}
		\noindent \textbf{Case 1: For $i=1$.}  This yields
		\begin{align*}
			\mathcal{H}_1 = \mathcal M_1 &= \{ \Theta_k\dots\Theta_2f_2 \oplus Z_k^*(\Delta_k\Theta_{k-1}\dots\Theta_2f_2 \oplus \dots \oplus \Delta_2f_2 \oplus g_1) \\
			&\quad : f_2 \in \gt\mathcal{E}_2, \, g_1 \in \ov{\im(\Delta_1)} \} \ominus \mathcal{G}_{\Theta}.
		\end{align*}
		Let $h_1 \in \mathcal{H}_1$. Then $h_1 = \Phi_2f_2 \oplus Z_k^*(\Lambda_2f_2 \oplus g_1)$ for some $f_2 \in \gt\mathcal{E}_2$ and $g_1 \in \ov{\im(\Delta_1)}$, satisfying the orthogonality condition
		\[
		\langle \Phi_2f_2 \oplus Z_k^*(\Lambda_2f_2 \oplus g_1), \Theta w \oplus \Delta_{\Theta}w \rangle = 0 \quad \text{for all } w \in \gt\mathcal{E}.
		\]
		This implies
		\begin{align*}
			\langle \Phi_2f_2, \Phi_1w \rangle + \langle \Lambda_2f_2 \oplus g_1, \Lambda_2\Theta_1w \oplus \Delta_1w \rangle &= 0 \\
			\langle \Phi_2^*\Phi_2f_2, \Theta_1w \rangle + \langle \Lambda_2^*\Lambda_2f_2, \Theta_1w \rangle + \langle g_1, \Delta_1w \rangle &= 0 \\
			\langle f_2, \Theta_1w \rangle + \langle g_1, \Delta_1w \rangle &= 0 \\
			\langle f_2 \oplus g_1, \Theta_1w \oplus \Delta_1w \rangle &= 0.
		\end{align*}
		This orthogonality guarantees that $f_2 \oplus g_1 \in \mathcal{H}(\Theta_1)$. Therefore, we can characterize $\mathcal{H}_1$ as
		\[
		\mathcal{H}_1 = \{ \Phi_2f_2 \oplus Z_k^*(\Lambda_2f_2 \oplus g_1) : f_2 \oplus g_1 \in \mathcal{H}(\Theta_1) \}.
		\]
		Define a mapping $U_1 : \mathcal{H}_1 \rightarrow \mathcal{H}(\Theta_1)$ by setting
		\[
		U_1(\Phi_2f_2 \oplus Z_k^*(\Lambda_2f_2 \oplus g_1)) = f_2 \oplus g_1.
		\]
		To prove that $U_1$ is a well-defined unitary operator, it suffices to show that it is an isometry
		\begin{align*}
			\| \Phi_2f_2 \oplus Z_k^*(\Lambda_2f_2 \oplus g_1) \|^2 &= \| \Phi_2f_2 \|^2 + \| \Lambda_2f_2 \|^2 + \| g_1 \|^2 \\
			&= \| f_2 \|^2 + \| g_1 \|^2 = \| f_2 \oplus g_1 \|^2.
		\end{align*}
		Next, for each \(j=1,\dots,n\), consider the action of the operator
		\begin{align} \label{eq:2.4}
			(U_1A_j^{1*}U_1^*)(f_2 \oplus g_1) &= (U_1P_{\mathcal{H}_1}T_{\Theta,j}^*|_{\mathcal{H}_1})(\Phi_2f_2 \oplus Z_k^*(\Lambda_2f_2 \oplus g_1)) \nonumber \\
			&= U_1P_{\mathcal{H}_1}T_{\Theta,j}^*(\Phi_2f_2 \oplus Z_k^*(\Lambda_2f_2 \oplus g_1)) \nonumber \\
			&= U_1P_{\mathcal{H}_1}\{ (S_j^* \otimes I_{\mathcal{E}_*})\Phi_2f_2 \oplus C_j^*(Z_k^*(\Lambda_2f_2 \oplus g_1)) \}.
		\end{align}
		Let $y = (S_j^* \otimes I_{\mathcal{E}_*})\Phi_2f_2 \oplus C_j^*(Z_k^*(\Lambda_2f_2 \oplus g_1))$. We utilize the standard orthogonal decomposition $f_2 = \sum_{r=1}^n(S_rS_r^* \otimes I_{\mathcal{E}_2})f_2 + f_2(0)$, where $f_2(0) = P_{e_\emptyset\otimes\mathcal{E}_2}f_2$.
		Then,
		\begin{align*}
			(S_j^* \otimes I_{\mathcal{E}_*})\Phi_2f_2 &= (S_j^* \otimes I_{\mathcal{E}_*})\Phi_2\left(\sum_{r=1}^n(S_rS_r^* \otimes I_{\mathcal{E}_2})f_2 + f_2(0)\right) \\
			&= (S_j^* \otimes I_{\mathcal{E}_*})\left(\sum_{r=1}^n(S_r \otimes I_{\mathcal{E}_k})\Phi_2(S_r^* \otimes I_{\mathcal{E}_2})f_2 + \Phi_2f_2(0)\right) \\
			&= \Phi_2(S_j^* \otimes I_{\mathcal{E}_2})f_2 + (S_j^* \otimes I_{\mathcal{E}_*})\Phi_2f_2(0).
		\end{align*}
		This gives the fundamental relation
		\begin{equation} \label{eq:2.5}
			(S_j^* \otimes I_{\mathcal{E}_*})\Phi_2f_2 = \Phi_2(S_j^* \otimes I_{\mathcal{E}_2})f_2 + (S_j^* \otimes I_{\mathcal{E}_*})\Phi_2f_2(0).
		\end{equation}
		We decompose $y = u_1 \oplus v_1$, where
		\[
		u_1 = \Phi_2(S_j^* \otimes I_{\mathcal{E}_2})f_2 \oplus Z_k^*(\Lambda_2(S_j^* \otimes I_{\mathcal{E}_2})f_2 \oplus C_j^{(1)*}g_1)
		\]
		and
		\[
		v_1 = (S_j^* \otimes I_{\mathcal{E}_*})\Phi_2f_2(0) \oplus [C_j^*(Z_k^*(\Lambda_2f_2 \oplus g_1)) - Z_k^*(\Lambda_2(S_j^* \otimes I_{\mathcal{E}_2})f_2 \oplus C_j^{(1)*}g_1)].
		\]
		Since $T_{\Theta_1,j}^*(f_2 \oplus g_1) = (S_j^* \otimes I_{\mathcal{E}_2})f_2 \oplus C_j^{(1)*}g_1 \in \mathcal{H}(\Theta_1)$, it is immediate that $u_1 \in \mathcal{H}_1$. 
		
		We now claim that $v_1 \perp \mathcal{H}_1$. By Lemma \ref{cuntz_isometry}, we know that $C_j^*(Z_k^*(0_{k-1} \oplus g_1)) = Z_k^*(0_{k-1} \oplus C_j^{(1)*}g_1)$. Therefore, $v_1$ simplifies to
		\[
		v_1 = (S_j^* \otimes I_{\mathcal{E}_*})\Phi_2f_2(0) \oplus [C_j^*(Z_k^*(\Lambda_2f_2 \oplus 0)) - Z_k^*(\Lambda_2(S_j^* \otimes I_{\mathcal{E}_2})f_2 \oplus 0)].
		\]
		To prove $v_1 \perp \mathcal{H}_1$, we must show that $\langle v_1, h_1 \rangle = 0$ for all $h_1 \in \mathcal{H}_1$. Any element $h_1 \in \mathcal{H}_1$ admits the representation
		\[
		h_1 = \Phi_2w_2 \oplus Z_k^*(\Lambda_2w_2 \oplus g_1), \quad \text{where } w_2 \oplus g_1 \in \mathcal{H}(\Theta_1).
		\]
		Define the second component of $v_1$ as
		\begin{equation} \label{eq:2.6}
			\Pi_2 = C_j^*(Z_k^*(\Lambda_2f_2 \oplus 0)) - Z_k^*(\Lambda_2(S_j^* \otimes I_{\mathcal{E}_2})f_2 \oplus 0).
		\end{equation}
		To simplify $\Pi_2$, we apply the intertwining relation $Z_k C_j = D_j Z_k$ from Lemma \ref{cuntz_isometry}, where $D_j = \diag(C_j^{(k)}, \dots, C_j^{(1)})$. Taking the adjoint yields $C_j^* Z_k^* = Z_k^* D_j^*$. Applying this identity to the first term of $\Pi_2$
		\begin{equation} \label{eq:2.7}
			C_j^*(Z_k^*(\Lambda_2f_2 \oplus 0)) = Z_k^*(D_{j \geq 2}^* \Lambda_2f_2 \oplus 0),
		\end{equation}
		where $D_{j \geq 2} = \diag(C_j^{(k)}, \dots, C_j^{(2)})$. 
		
		Let $Z_{k-1}^{\{k\},\dots,\{2\}}$ be the canonical isometry associated with the sub-factorization $\Phi_2 = \Theta_k \dots \Theta_2$, satisfying $\Lambda_2f_2 = Z_{k-1}^{\{k\},\dots,\{2\}} \Delta_{\Phi_2}f_2$. Applying Lemma \ref{cuntz_isometry} directly to the factorization $\Phi_2 = \Theta_k \dots \Theta_2$, we deduce that
		\[
		Z_{k-1}^{\{k\},\dots,\{2\}} C_j^{(\Phi_2)} = D_{j \geq 2} Z_{k-1}^{\{k\},\dots,\{2\}}.
		\]
		From Proposition \ref{propo_k_regular_multi} the $k$-regular factorization of $\Theta=\Theta_k\dots\Theta_1$ implies that $\Phi_2 = \Theta_k \dots \Theta_2$ is a $(k-1)$-regular factorization, $Z_{k-1}^{\{k\},\dots,\{2\}}$ is a unitary operator. This yields
		\[
		Z_{k-1}^{\{k\},\dots,\{2\}} C_j^{(\Phi_2)*} = D_{j \geq 2}^* Z_{k-1}^{\{k\},\dots,\{2\}}.
		\]
		Consequently, from equation \eqref{eq:2.7}, we get
		\begin{align*}
			C_j^*(Z_k^*(\Lambda_2f_2 \oplus 0)) &= Z_k^*(D_{j \geq 2}^* Z_{k-1}^{\{k\},\dots,\{2\}} \Delta_{\Phi_2}f_2 \oplus 0) \\
			&= Z_k^*(Z_{k-1}^{\{k\},\dots,\{2\}} C_j^{(\Phi_2)*} \Delta_{\Phi_2}f_2 \oplus 0).
		\end{align*}
		Substituting this back into $\Pi_2$, we find
		\begin{align*}
			\Pi_2 &= Z_k^*(Z_{k-1}^{\{k\},\dots,\{2\}} C_j^{(\Phi_2)*} \Delta_{\Phi_2}f_2 \oplus 0) - Z_k^*(Z_{k-1}^{\{k\},\dots,\{2\}} \Delta_{\Phi_2}(S_j^* \otimes I_{\mathcal{E}_2})f_2 \oplus 0) \\
			&= Z_k^*(Z_{k-1}^{\{k\},\dots,\{2\}} x_1 \oplus 0),
		\end{align*}
		where $x_1 = C_j^{(\Phi_2)*} \Delta_{\Phi_2}f_2 - \Delta_{\Phi_2}(S_j^* \otimes I_{\mathcal{E}_2})f_2$.
		
		Next, utilizing the fact that $Z_k$ is a unitary operator, we evaluate the inner product
		\begin{align*}
			\langle \Pi_2, Z_k^*(\Lambda_2w_2 \oplus g_1) \rangle &= \langle Z_k^*(Z_{k-1}^{\{k\},\dots,\{2\}} x_1 \oplus 0), Z_k^*(\Lambda_2w_2 \oplus g_1) \rangle \\
			&= \langle Z_{k-1}^{\{k\},\dots,\{2\}} x_1, \Lambda_2w_2 \rangle \\
			&= \langle x_1, \Delta_{\Phi_2}w_2 \rangle \quad (\text{since } \Lambda_2w_2 = Z_{k-1}^{\{k\},\dots,\{2\}} \Delta_{\Phi_2}w_2).
		\end{align*}
		Using the relation $x_1 = C_j^{(\Phi_2)*} \Delta_{\Phi_2}f_2 - \Delta_{\Phi_2}(S_j^* \otimes I_{\mathcal{E}_2})f_2,$ we get 
		\begin{align*}
			\langle x_1, \Delta_{\Phi_2}w_2 \rangle &= \langle \Delta_{\Phi_2}f_2, C_j^{(\Phi_2)} \Delta_{\Phi_2}w_2 \rangle - \langle \Delta_{\Phi_2}(S_j^* \otimes I_{\mathcal{E}_2})f_2, \Delta_{\Phi_2}w_2 \rangle \\
			&= \langle \Delta_{\Phi_2}f_2, \Delta_{\Phi_2}(S_j \otimes I_{\mathcal{E}_2})w_2 \rangle - \langle \Delta_{\Phi_2}(S_j^* \otimes I_{\mathcal{E}_2})f_2, \Delta_{\Phi_2}w_2 \rangle \\
			&= \langle f_2, (S_j \otimes I_{\mathcal{E}_2})w_2 \rangle - \langle \Phi_2f_2, \Phi_2(S_j \otimes I_{\mathcal{E}_2})w_2 \rangle \\
			&\quad - \left( \langle (S_j^* \otimes I_{\mathcal{E}_2})f_2, w_2 \rangle - \langle \Phi_2(S_j^* \otimes I_{\mathcal{E}_2})f_2, \Phi_2w_2 \rangle \right) \\
			&= \langle \Phi_2(S_j^* \otimes I_{\mathcal{E}_2})f_2, \Phi_2w_2 \rangle - \langle (S_j^* \otimes I_{\mathcal{E}_*})\Phi_2f_2, \Phi_2w_2 \rangle \\
			&= \langle (\Phi_2(S_j^* \otimes I_{\mathcal{E}_2}) - (S_j^* \otimes I_{\mathcal{E}_*})\Phi_2)f_2, \Phi_2w_2 \rangle.
		\end{align*}
		From the relation \eqref{eq:2.5}, we obtain $\Phi_2(S_j^* \otimes I_{\mathcal{E}_2})f_2 - (S_j^* \otimes I_{\mathcal{E}_*})\Phi_2f_2 = -(S_j^* \otimes I_{\mathcal{E}_*})\Phi_2f_2(0)$, so
		\[
		\langle x_1, \Delta_{\Phi_2}w_2 \rangle = -\langle (S_j^* \otimes I_{\mathcal{E}_*})\Phi_2f_2(0), \Phi_2w_2 \rangle.
		\]
		Finally, we compute $\langle v_1, h_1 \rangle$
		\begin{align*}
			\langle v_1, h_1 \rangle &= \langle (S_j^* \otimes I_{\mathcal{E}_*})\Phi_2f_2(0), \Phi_2w_2 \rangle + \langle \Pi_2, Z_k^*(\Lambda_2w_2 \oplus g_1) \rangle \\
			&= \langle (S_j^* \otimes I_{\mathcal{E}_*})\Phi_2f_2(0), \Phi_2w_2 \rangle - \langle (S_j^* \otimes I_{\mathcal{E}_*})\Phi_2f_2(0), \Phi_2w_2 \rangle \\
			&= 0.
		\end{align*}
		Therefore, $v_1 \perp \mathcal{H}_1$. Returning to equation \eqref{eq:2.4}, we obtain
		\begin{align*}
			(U_1A_j^{1*}U_1^*)(f_2 \oplus g_1) &= U_1P_{\mathcal{H}_1} \{u_1 \oplus v_1\} \\
			&= U_1(u_1) \\
			&= U_1\left( \Phi_2(S_j^* \otimes I_{\mathcal{E}_2})f_2 \oplus Z_k^*(\Lambda_2(S_j^* \otimes I_{\mathcal{E}_2})f_2 \oplus C_j^{(1)*}g_1) \right) \\
			&= (S_j^* \otimes I_{\mathcal{E}_2})f_2 \oplus C_j^{(1)*}g_1 \\
			&= T_{\Theta_1,j}^*(f_2 \oplus g_1).
		\end{align*}
		Therefore, $A^1 = [A_1^1, \dots, A_n^1]$ is unitarily equivalent to the model operator $T_{\Theta_1} = [T_{\Theta_1,1}, \dots, T_{\Theta_1,n}]$.
		
		\vspace{0.3cm}
		\noindent \textbf{Case 2: For $i=2, \dots, k-1$.}
		From Theorem \ref{main_thereom_02}, the subspace $\mathcal{H}_i = \mathcal M_i \ominus \mathcal M_{i-1}$ is given as 
		\begin{align*}
			\mathcal{H}_i &= \{ \Phi_{i+1}f_{i+1} \oplus Z_k^*(\Lambda_{i+1}f_{i+1} \oplus g_i \oplus 0_{i-1}) : f_{i+1} \in \gt\mathcal{E}_{i+1}, \, g_i \in \ov{\im(\Delta_i)} \} \\
			&\quad\ominus \{{\Phi_if_i \oplus Z_k^*(\Lambda_if_i \oplus 0_{i-1}):f_i \in \gt\mathcal{E}_i}\}.
		\end{align*}
		Let $h_i \in \mathcal{H}_i$. It takes the form
		\[
		h_i = \Phi_{i+1}f_{i+1} \oplus Z_k^*(\Lambda_{i+1}f_{i+1} \oplus g_i \oplus 0_{i-1})
		\]
		such that for all $f_i \in \gt\mathcal{E}_i$,
		\begin{equation} \label{eq:2.9}
			\langle \Phi_{i+1}f_{i+1} \oplus Z_k^*(\Lambda_{i+1}f_{i+1} \oplus g_i \oplus 0_{i-1}), \Phi_if_i \oplus Z_k^*(\Lambda_if_i \oplus 0_{i-1}) \rangle = 0.
		\end{equation}
		This yields
		\begin{align*}
			\langle \Phi_{i+1}f_{i+1}, \Phi_if_i \rangle + \langle \Lambda_{i+1}f_{i+1} \oplus g_i, \Lambda_if_i \rangle &= 0 \\
			\langle \Phi_{i+1}^*\Phi_{i+1}f_{i+1}, \Theta_if_i \rangle + \langle \Lambda_{i+1}^*\Lambda_{i+1}f_{i+1}, \Theta_if_i \rangle + \langle g_i, \Delta_if_i \rangle &= 0 \\
			\langle f_{i+1} \oplus g_i, \Theta_if_i \oplus \Delta_if_i \rangle &= 0.
		\end{align*}
		This orthogonality condition ensures that $f_{i+1} \oplus g_i \in \mathcal{H}(\Theta_i)$. Therefore, $\mathcal{H}_i$ is precisely characterized as
		\[
		\mathcal{H}_i = \{ \Phi_{i+1}f_{i+1} \oplus Z_k^*(\Lambda_{i+1}f_{i+1} \oplus g_i \oplus 0_{i-1}) : f_{i+1} \oplus g_i \in \mathcal{H}(\Theta_i) \}.
		\]
		We define the mapping $U_i : \mathcal{H}_i \rightarrow \mathcal{H}(\Theta_i)$ by
		\[
		U_i(\Phi_{i+1}f_{i+1} \oplus Z_k^*(\Lambda_{i+1}f_{i+1} \oplus g_i \oplus 0_{i-1})) = f_{i+1} \oplus g_i.
		\]
		This operator is surjective by definition, and it acts as an isometry because
		\begin{align*}
			\| \Phi_{i+1}f_{i+1} \oplus Z_k^*(\Lambda_{i+1}f_{i+1} \oplus g_i \oplus 0_{i-1}) \|^2 &= \| \Phi_{i+1}f_{i+1} \|^2 + \| \Lambda_{i+1}f_{i+1} \|^2 + \| g_i \|^2 \\
			&= \| f_{i+1} \|^2 + \| g_i \|^2 = \| f_{i+1} \oplus g_i \|^2.
		\end{align*}
		Hence, $U_i$ is a well-defined unitary operator mapping the $i$-th orthogonal difference space exactly to the functional model space of $\Theta_i$. Now we compute
		\begin{align} \label{eq:3.0}
			(U_iA_j^{i*}U_i^*)(f_{i+1} \oplus g_i) &= U_i P_{\mathcal{H}_i} T_{\Theta,j}^* \{ \Phi_{i+1}f_{i+1} \oplus Z_k^*(\Lambda_{i+1}f_{i+1} \oplus g_i \oplus 0_{i-1}) \} \nonumber \\
			&= U_i P_{\mathcal{H}_i} \{ (S_j^* \otimes I_{\mathcal{E}_*})\Phi_{i+1}f_{i+1} \oplus C_j^*Z_k^*(\Lambda_{i+1}f_{i+1} \oplus g_i \oplus 0_{i-1}) \}.
		\end{align}
		Let $y = (S_j^* \otimes I_{\mathcal{E}_*})\Phi_{i+1}f_{i+1} \oplus C_j^*Z_k^*(\Lambda_{i+1}f_{i+1} \oplus g_i \oplus 0_{i-1})$. We split $y = u_i \oplus v_i$ where
		\[
		u_i = \Phi_{i+1}(S_j^* \otimes I_{\mathcal{E}_{i+1}})f_{i+1} \oplus Z_k^*(\Lambda_{i+1}(S_j^* \otimes I_{\mathcal{E}_{i+1}})f_{i+1} \oplus C_j^{(i)*}g_i \oplus 0_{i-1})
		\]
		and
		\begin{align*}
			v_i &= (S_j^* \otimes I_{\mathcal{E}_*})\Phi_{i+1}f_{i+1}(0) \oplus [ C_j^*Z_k^*(\Lambda_{i+1}f_{i+1} \oplus g_i \oplus 0_{i-1}) \\
			&\quad - Z_k^*(\Lambda_{i+1}(S_j^* \otimes I_{\mathcal{E}_{i+1}})f_{i+1} \oplus C_j^{(i)*}g_i \oplus 0_{i-1}) ].
		\end{align*}
		Since $T_{\Theta_i,j}^*(f_{i+1} \oplus g_i) = (S_j^* \otimes I_{\mathcal{E}_{i+1}})f_{i+1} \oplus C_j^{(i)*}g_i \in \mathcal{H}(\Theta_i)$, this establishes that $u_i \in \mathcal{H}_i$.
		
		We claim that $v_i \perp \mathcal{H}_i$. By Lemma \ref{cuntz_isometry}, we have
		\[
		C_j^* Z_k^*(0_{k-i} \oplus g_i \oplus 0_{i-1})
		= Z_k^*\bigl(0_{k-i} \oplus {C_j^{(i)}}^* g_i \oplus 0_{i-1}\bigr),
		\]
		which yields the following simplified expression for $v_i$
		\begin{align*}
			v_i &= (S_j^* \otimes I_{\mathcal{E}_*})\Phi_{i+1}f_{i+1}(0) \oplus [ C_j^*Z_k^*(\Lambda_{i+1}f_{i+1} \oplus 0_i)- Z_k^*(\Lambda_{i+1}(S_j^* \otimes I_{\mathcal{E}_{i+1}})f_{i+1} \oplus 0_i) ].
		\end{align*}
		For any $h_i \in \mathcal{H}_i$, $h_i$ admits the representation
		\[
		h_i = \Phi_{i+1}w_{i+1} \oplus Z_k^*(\Lambda_{i+1}w_{i+1} \oplus g_i \oplus 0_{i-1}), \quad \text{where } w_{i+1} \oplus g_i \in \mathcal{H}(\Theta_i).
		\]
		Let $\Pi_{i+1}$ denote the second component of $v_i$. Then, using the relation $C_j^* Z_k^* = Z_k D_j^*$, we obtain
		\begin{align*}
			\Pi_{i+1} &= C_j^*Z_k^*(\Lambda_{i+1}f_{i+1} \oplus 0_i) - Z_k^*(\Lambda_{i+1}(S_j^* \otimes I_{\mathcal{E}_{i+1}})f_{i+1} \oplus 0_i) \\
			&= Z_k^*(D_{j \geq i+1}^* \Lambda_{i+1}f_{i+1} \oplus 0_i) - Z_k^*(\Lambda_{i+1}(S_j^* \otimes I_{\mathcal{E}_{i+1}})f_{i+1} \oplus 0_i),
		\end{align*}
		where $D_{j \geq i+1} = \diag(C_j^{(k)}, \dots, C_j^{(i+1)})$. Let $Z_{k-i}^{\{k\},\dots,\{i+1\}}$ be the canonical isometry for the sub-factorization $\Phi_{i+1} = \Theta_k \dots \Theta_{i+1}$. By Lemma \ref{cuntz_isometry}, 
		\[
		Z_{k-i}^{\{k\},\dots,\{i+1\}} C_j^{(\Phi_{i+1})} = D_{j \geq i+1} Z_{k-i}^{\{k\},\dots,\{i+1\}}.
		\]
		Because from Proposition \ref{propo_k_regular_multi} the sub-factorization $\Phi_{i+1} = \Theta_k \dots \Theta_{i+1}$ is regular, this isometry is unitary,and using the relations $Z_{k-i}^{\{k\},\dots,\{i+1\}} {C_j^{(\Phi_{i+1})}}^* = D^*_{j \geq i+1} Z_{k-i}^{\{k\},\dots,\{i+1\}} $ and  $\Lambda_{i+1} = Z_{k-i}^{\{k\},\dots,\{i+1\}} \Delta_{\Phi_{i+1}}$ , we get 
		\[
		\Pi_{i+1} = Z_k^*(Z_{k-i}^{\{k\},\dots,\{i+1\}} x_i \oplus 0_i),
		\]
		where $x_i = C_j^{(\Phi_{i+1})*} \Delta_{\Phi_{i+1}}f_{i+1} - \Delta_{\Phi_{i+1}}(S_j^* \otimes I_{\mathcal{E}_{i+1}})f_{i+1}$.
		Next, utilizing $\Lambda_{i+1} = Z_{k-i}^{\{k\},\dots,\{i+1\}} \Delta_{\Phi_{i+1}}$, we calculate the inner product
		\begin{align} \label{eq:3.2}
			\langle \Pi_{i+1}, Z_k^*(\Lambda_{i+1}w_{i+1} \oplus g_i \oplus 0_{i-1}) \rangle &= \langle Z_{k-i}^{\{k\},\dots,\{i+1\}} x_i, \Lambda_{i+1}w_{i+1} \rangle \nonumber \\
			&= \langle x_i, \Delta_{\Phi_{i+1}}w_{i+1} \rangle \nonumber \\
			&= \left\langle \Delta_{\Phi_{i+1}} f_{i+1}, \, C_j^{(\Phi_{i+1})} \Delta_{\Phi_{i+1}} w_{i+1} \right\rangle \nonumber\\
			&
			- \left\langle \Delta_{\Phi_{i+1}} (S_j^* \otimes I_{\mathcal E_{i+1}}) f_{i+1}, \, \Delta_{\Phi_{i+1}} w_{i+1} \right\rangle \nonumber\\
			&= \langle (\Phi_{i+1}(S_j^* \otimes I_{\mathcal{E}_{i+1}}) - (S_j^* \otimes I_{\mathcal{E}_*})\Phi_{i+1})f_{i+1}, \Phi_{i+1}w_{i+1} \rangle \nonumber \\
			&= -\langle (S_j^* \otimes I_{\mathcal{E}_*})\Phi_{i+1}f_{i+1}(0), \Phi_{i+1}w_{i+1} \rangle.
		\end{align}
		Thus, completing the inner product for $v_i$, we obtain
		\begin{align*}
			\langle v_i, h_i \rangle &= \langle (S_j^* \otimes I_{\mathcal{E}_*})\Phi_{i+1}f_{i+1}(0), \Phi_{i+1}w_{i+1} \rangle + \langle \Pi_{i+1}, Z_k^*(\Lambda_{i+1}w_{i+1} \oplus g_i \oplus 0_{i-1}) \rangle \\
			&= 0.
		\end{align*}
		Hence, $v_i \perp \mathcal{H}_i$. Returning to equation \eqref{eq:3.0}, the projection isolates $u_i$, giving
		\begin{align*}
			(U_iA_j^{i*}U_i^*)(f_{i+1} \oplus g_i) &= U_i P_{\mathcal{H}_i} \{ u_i \oplus v_i \} = U_i(u_i) \\
			&= (S_j^* \otimes I_{\mathcal{E}_{i+1}})f_{i+1} \oplus C_j^{(i)*}g_i \\
			&= T_{\Theta_i,j}^*(f_{i+1} \oplus g_i).
		\end{align*}
		Therefore, for $i=2, \dots, k-1$, the row contraction $A^i = [A_1^i, \dots, A_n^i]$ is unitarily equivalent to $T_{\Theta_i}$.
		
		%%%%%%%%%%%%%%%%%%%%%%%%%%%%%%%%%%%%%%%%%%%%%%%%%%%%%%%%%%%%%%%%%%%%%%%%%%%%%%%%%%%%%%%%%%%%%%%%%%%%%%%%%%%%%%%%%%
		\noindent \textbf{Case 3.} For $i=k$, the subspace $\mathcal H_k$ is given by the orthogonal difference:
		\[
		\mathcal H_k = \mathcal H_{\Theta} \ominus \mathcal M_{k-1}.
		\]
		Using the structural decomposition, this space can be explicitly written as:
		\begin{align*}
			\mathcal H_k &= \left[ (\gt\mathcal{E}_{k+1}) \oplus Z_k^* (\ov{\Delta_k(\gt\mathcal{E}_k)} \oplus \{0\}\cdots\oplus\{0\}) \right] \\
			&\quad \ominus \{ \Theta_k f_k \oplus Z_k^*(\Delta_k f_k \oplus 0_{k-1}) : f_k \in \gt\mathcal{E}_k \}.
		\end{align*}
		
		If $h_k \in \mathcal{H}_k$, then $h_k = f_{k+1} \oplus Z_k^*(g_k \oplus 0_{k-1})$, and it must satisfy the orthogonality condition $\langle h_k, \Theta_k f_k \oplus Z_k^*(\Delta_k f_k \oplus 0_{k-1}) \rangle = 0$ for all $f_k \in \gt\mathcal{E}_k$.
		Because $Z_k^*$ is an isometry, this simplifies directly to:
		\[
		\langle f_{k+1} \oplus g_k, \Theta_k f_k \oplus \Delta_k f_k \rangle = 0, \quad \text{for all } f_k \in \gt\mathcal{E}_k.
		\]
		Thus, we can completely characterize $\mathcal H_k$ as
		\[
		\mathcal{H}_k = \{ f_{k+1} \oplus Z_k^*(g_k \oplus 0_{k-1}) : f_{k+1} \oplus g_k \in \mathcal{H}(\Theta_k) \}.
		\]
		We define the unitary operator $U_k : \mathcal{H}_k \rightarrow \mathcal{H}(\Theta_k)$ by
		\[
		U_k(f_{k+1} \oplus Z_k^*(g_k \oplus 0_{k-1})) = f_{k+1} \oplus g_k.
		\]
		As subspace $\mathcal M_k$ are invarinat under  $T_{\Theta,j}$ for each $j$, then $${A_j^k}^*=T_{\Theta,j}^*|_{\mathcal H_k}.$$
		We evaluate this explicitly
		\begin{align*}
			(U_k{A_j^{k}}^*U_k^*)(f_{k+1} \oplus g_k) &= U_k  T_{\Theta,j}^* (f_{k+1} \oplus Z_k^*(g_k \oplus 0_{k-1})) \\
			&= U_k  \{ (S_j^* \otimes I_{\mathcal{E}_{*}})f_{k+1} \oplus C_j^*Z_k^*(g_k \oplus 0_{k-1}) \} \\
			&= U_k  \{ (S_j^* \otimes I_{\mathcal{E}_{*}})f_{k+1} \oplus Z_k^*(C_j^{(k)*}g_k \oplus 0_{k-1}) \}\\
			&=(S_j^* \otimes I_{\mathcal{E}_{*}})f_{k+1} \oplus C_j^{(k)*}g_k\\
			&= T_{\Theta_k,j}^*(f_{k+1} \oplus g_k).
		\end{align*}
		This confirms that $A^k = [A_1^k, \dots, A_n^k]$ is unitarily equivalent to $T_{\Theta_k}$, concluding the entire proof.
	\end{proof}

	\begin{corollary}
		Under the hypotheses of Theorem $\ref{unitary_equivalent_in_multivarible}$, the characteristic function of the row contraction $A_i = [A_1^i, \dots, A_n^i]$ coincides with the purely contractive part of the multi-analytic operator $\Theta_i$ for each $i = 1, \dots, k$.
	\end{corollary}
	
	\begin{corollary}
		Assume the hypotheses of Theorem{\rm ~\ref{unitary_equivalent_in_multivarible}}. Then, for each $i = 1, \dots, k-2$, the equality $\mathcal{M}_i = \mathcal{M}_{i+1}$ holds if and only if $\Theta_{i+1}$ is a unitary constant; that is, there exists a unitary operator $U : \mathcal{E}_{i+1} \to \mathcal{E}_{i+2}$ such that 
		\[
		\Theta_{i+1} = I_{\Gamma} \otimes U.
		\]
	\end{corollary}
	\begin{proof}
		For each $i = 1, \dots, k-2$, the subspace 
		\[
		\mathcal{H}_{i+1} = \mathcal{M}_{i+1} \ominus \mathcal{M}_i
		\]
		is unitarily equivalent to the model space $H(\Theta_{i+1})$. Moreover, $H(\Theta_{i+1}) = \{0\}$ if and only if $\Theta_{i+1}$ is a unitary constant. The result now follows.
	\end{proof}
	
	It is worth noting that, apart from the regularity aspect, part~(i) of the following theorem was already established by G.~Popescu in Theorem~3.8 of \cite{Po06}. The following theorem strengthens the preceding result.
	
	\begin{theorem}\label{m_r_divisor}
		Let $\Theta:\Gamma\otimes\mathcal E \to \Gamma\otimes\mathcal E_*$ be a purely contractive multi-analytic operator satisfying the Szeg\H{o} condition. Then we have the following results
		
		\begin{enumerate}[label=\rm(\roman*), ref=\thetheorem(\roman*)]
			\item \label{m_r_divisor_a}
			Let $\Theta = \Theta_2 \Theta_1 = \Theta'_2 \Theta'_1$ be $2$-regular factorizations, and let $\mathcal{M}$ and $\mathcal{M}'$ be the joint invariant subspaces under $T_\Theta$ corresponding to these factorizations, respectively. If $\mathcal{M} \subset \mathcal{M}'$, then there exists a contractive multi-analytic operator $\Omega$ such that
			\begin{align}\label{divisor}
				\Theta'_1 = \Omega \Theta_1.
			\end{align}
			Moreover, $\Theta_2 = \Theta'_2 \Omega$, and the factorization in \eqref{divisor} is $2$-regular. In addition, if $\mathcal{M} = \mathcal{M}'$, then $\Omega$ is a unitary constant operator.
			\item \label{m_r_divisor_b} Suppose that 
			\[
			\Theta = \Theta'_{2} \Theta'_{1} \quad \text{and} \quad \Theta'_1 = \Omega \Theta_{1}
			\]
			are $2$-regular factorizations. Consequently, the product $\Theta = (\Theta'_{2} \Omega) \Theta_{1}$ yields a $2$-regular factorization. Let $\mathcal M$ and $\mathcal M'$ denote the joint invariant subspaces under $T_\Theta$   corresponding to the $2$-regular factorizations $\Theta = (\Theta'_{2} \Omega) \Theta_{1}$ and $\Theta = \Theta'_{2} \Theta'_{1}$, respectively. Then, we have $\mathcal M \subset \mathcal M'$.
		\end{enumerate}
	\end{theorem}
	\begin{proof} (i) In the Theorem 3.8 of article \cite{Po06}  G. Popescu has proved that there exist multi-analytic operator $\Omega$ such that                 
		\[
		\Theta'_1 = \Omega \Theta_1 
		\] and if $\mathcal M = \mathcal M'$, then $\Omega$ is a unitary constant operator. In the proof of this theorem he also prove that $\Theta_2 = \Theta'_2 \Omega$. It remains to show that	\[
		\Theta'_1 = \Omega \Theta_1 
		\] is a $2$-regular factorization. Since the factorizations \(\Theta = (\Theta_2' \Omega) \Theta_1\) and \(\Theta = \Theta_2' (\Omega \Theta_1)\) are \(2\)-regular, it follows from the Proposition \ref{propo_k_regular_multi} that the factorization \(\Theta = \Theta_2' \Omega \Theta_1\) is \(3\)-regular. Consequently, the factorization
		\(\Theta'_1 = \Omega \Theta_1\) is \(2\)-regular.\\
		(ii) Let \(\mathcal{M}\) and \(\mathcal{M}'\) be the joint invariant subspaces associated with the $2$-regular factorizations \(\Theta = \Theta_2 \Theta_1\) and \(\Theta = \Theta_2' \Theta_1'\), respectively with \(\Theta_2 = \Theta_2' \Omega\). It follows from Theorem \ref{2_r_invariant} that these subspaces are given by
		
		\begin{align*}
			{\mathcal M} = \bigg\{ \Theta_2 f \oplus Z^*(\Delta_2 f \oplus g) : \ & f \in \Gamma \otimes \mathcal{F}, \ g \in \overline{\Delta_1 (\Gamma \otimes \mathcal{E})} \bigg\} \\
			& \ominus \bigg\{ \Theta w \oplus \Delta_\Theta w : w \in \Gamma \otimes \mathcal{E} \bigg\},\\
			{\mathcal M}' = \bigg\{ \Theta'_2 f' \oplus Z'^*(\Delta'_2 f' \oplus g') : \ & f' \in \Gamma \otimes \mathcal{F'}, \ g' \in \overline{\Delta'_1 (\Gamma \otimes \mathcal{E})} \bigg\} \\
			& \ominus \bigg\{ \Theta w \oplus \Delta_\Theta w : w \in \Gamma \otimes \mathcal{E} \bigg\},\\
		\end{align*}
		where \(Z\) and \(Z'\) are unitary operators associated with $2$-regular factorizations of \(\Theta\): \(\Theta = \Theta_2 \Theta_1\) and \(\Theta = \Theta_2' \Theta_1'\), respectively. Assume \(\Theta_2 f \oplus Z^*(\Delta_2 f \oplus g) \in \mathcal{M}\) for some \(f \in \Gamma \otimes \mathcal{F}\) and \(g \in \overline{\Delta_1 (\Gamma \otimes \mathcal{E})}\). Let \(f' = \Omega f\), which implies that \(\Theta_2 f = \Theta_2' f'\).
		Observe that \(\Theta_2 = \Theta_2' \Omega\) is $2$-regular. Let \(Z^{''}:\overline{\Delta_2(\Gamma \otimes \mathcal{F})} \to \overline{\Delta'_2(\Gamma \otimes \mathcal{F'})} \oplus \overline{\Delta_{\Omega}(\Gamma \otimes \mathcal{F})}\) be the associated unitary operator defined by 
		
		\[
		Z''(\Delta_2 f) := \Delta'_2 \Omega f \oplus \Delta_{\Omega} f, \quad f \in \Gamma \otimes \mathcal{F}.
		\]
	Then
		\begin{align*}
			Z^*(\Delta_2 f \oplus g) &= Z^*(Z''^* \oplus I_{\overline{\Delta_1 (\Gamma \otimes \mathcal{E})}})((\Delta'_2 \Omega f \oplus \Delta_{\Omega} f) \oplus g) \\
			&= Z'^*(I_{\overline{\Delta'_2(\Gamma \otimes \mathcal{F'})}} \oplus Z'''^*)(\Delta'_2 \Omega f \oplus \Delta_{\Omega} f \oplus g) \\
			&= Z'^*(\Delta'_2 f' \oplus g'),
		\end{align*}
		
		where \(Z'''\) is the associated unitary operator corresponding to the $2$-regular factorization \(\Theta'_1 = \Omega \Theta_1\), and \(g' = Z'''^*(\Delta_{\Omega} f \oplus g\)). Hence, we conclude that \(\mathcal{M} \subset \mathcal{M}'\).
	\end{proof}

	\begin{proposition}
		Let $\Theta = \Theta_k \cdots \Theta_1$ be a $k$-regular factorization of a purely contractive multi-analytic operator satisfying the Szeg\H{o} condition, where \( {\Theta}_i:\Gamma\otimes \mathcal{E}_i\to \Gamma\otimes \mathcal{E}_{i+1}  \) are contractive multi-analytic operators for \( i=1,\cdots,k \) with \( \mathcal{E}_1 = \mathcal{E} \) and \( \mathcal{E}_{k+1} = \mathcal{E}_* \) and let 
		\[
		\mathcal{M}_1 \sub \mathcal{M}_2 \sub \cdots \sub \mathcal{M}_{k-1}
		\]
		be the corresponding chain of joint invariant subspaces for the model contraction $T_\Theta$ acting on $\mathcal{H}(\Theta)$. Partition the index set $\{1, \dots, k\}$ into $r$ disjoint subsets $J_1, \dots, J_r$ defined by
		\[
		J_1 = \{ j_1, \ldots, 1 \}, \quad 
		J_i = \{j_i, \dots, j_{i-1}+1\} \ (i=2,\dots,r),
		\]
		where $1 \le j_1 < j_2 < \cdots < j_r = k$. 
		
		Consider the associated aggregated $r$-regular factorization
		\[
		\Theta = \Theta_{J_r} \cdots \Theta_{J_1}, 
		\quad \text{where} \quad 
		\Theta_{J_i} := \Theta_{j_i} \cdots \Theta_{j_{i-1}+1}.
		\]
		Let $\mathcal{M}_{J_1} \sub \cdots \sub \mathcal{M}_{J_{r-1}}$ be the corresponding chain of joint invariant subspaces. Then
		\[
		\mathcal{M}_{J_i} = \mathcal{M}_{j_i}, \quad \text{for all } i = 1, \dots, r-1.
		\]
	\end{proposition}
	\begin{proof}
		Fix an index $i \in \{1, \dots, r-1\}$; by the definition of the partition, we clearly have $j_i < k$. By Theorem \ref{main_thereom_02}, the joint invariant subspace $\mathcal{M}_{j_i}$ corresponding to the $k$-regular factorization $\Theta = \Theta_k \cdots \Theta_1$ admits the explicit representation
		\begin{align}\label{eq:M_ji_pdf}
			\mathcal{M}_{j_i} &= \Big\{ \Theta_k \cdots \Theta_{j_i+1} f \oplus Z_k^* \big( \Delta_k \Theta_{k-1} \cdots \Theta_{j_i+1} f \oplus \cdots \oplus \Delta_{j_i+1} f \oplus g_{j_i} \oplus \cdots \oplus g_1 \big) : \nonumber \\
			&\qquad f \in \gt\mathcal{E}_{j_i+1}, \; g_l \in \ov{\im(\Delta_l)}, \text{ for } l = 1, \dots, j_i \Big\} \ominus \mathcal{G}_\Theta,
		\end{align}
		where $\mathcal{G}_\Theta = \{ \Theta f \oplus \Delta_\Theta f : f \in \gt\mathcal{E} \}$.
		
		Applying the same theorem to the aggregated $r$-regular factorization $\Theta = \Theta_{J_r} \cdots \Theta_{J_1}$, the corresponding joint invariant subspace $\mathcal{M}_{J_i}$ is given by
		\begin{align}\label{eq:M_Ji_pdf}
			\mathcal{M}_{J_i} &= \Big\{ \Theta_{J_r} \cdots \Theta_{J_{i+1}} f \oplus (Z_r^{J_r, \dots, J_1})^* \big( \Delta_{J_r} \Theta_{J_{r-1}} \cdots \Theta_{J_{i+1}} f \oplus \cdots \oplus \Delta_{J_{i+1}} f \oplus w_i \oplus \cdots \oplus w_1 \big) : \nonumber \\
			&\qquad f \in \gt\mathcal{E}_{j_i+1}, \; w_m \in \ov{\im(\Delta_{J_m})}, \text{ for } m = 1, \dots, i \Big\} \ominus \mathcal{G}_\Theta,
		\end{align}
		where $\Delta_{J_m} = (I - \Theta_{J_m}^* \Theta_{J_m})^{1/2}$.
		
		Recall the identity established in \eqref{deco_relation_ZK} in the proof of Proposition \ref{partition_multi}:
		\[
		Z_k = \left( \bigoplus_{m=r}^1 Z_{|J_m|}^{\{j_m\}, \dots, \{j_{m-1}+1\}} \right) Z_r^{J_r, \dots, J_1}.
		\]
		Since all operators in this decomposition are unitary maps, it follows that
		\[
		(Z_r^{J_r, \dots, J_1})^* = Z_k^* \left( \bigoplus_{m=r}^1 Z_{|J_m|}^{\{j_m\}, \dots, \{j_{m-1}+1\}} \right).
		\]
		
		Next, consider the action of $(Z_r^{J_r, \dots, J_1})^*$ on the components of the defect space. Observe that
		\begin{align*}
			&\big( Z_{|J_m|}^{\{j_m\}, \dots, \{j_{m-1}+1\}} \big) \big( \Delta_{J_m} \Theta_{J_{m-1}} \cdots \Theta_{J_{i+1}} f \big) \\
			&\quad = \Delta_{j_m} \Theta_{j_m-1} \cdots \Theta_{j_i+1} f \oplus \cdots \oplus \Delta_{j_{m-1}+1} \Theta_{j_{m-1}} \cdots \Theta_{j_i+1} f.
		\end{align*}
		For $m = i, \dots, 1$, and an arbitrary vector $w_m \in \ov{\im(\Delta_{J_m})}$, the following unitary identification holds
		\[
		Z_{|J_m|}^{\{j_m\}, \dots, \{j_{m-1}+1\}}( \ov{\im(\Delta_{J_m})} ) = \ov{\im(\Delta_{j_m})} \oplus \cdots \oplus \ov{\im(\Delta_{j_{m-1}+1})}.
		\]
		Using these relations, we obtain
		\begin{align*}
			&(Z_r^{J_r, \dots, J_1})^* \big( \Delta_{J_r} \Theta_{J_{r-1}} \cdots \Theta_{J_{i+1}} f \oplus \cdots \oplus \Delta_{J_{i+1}} f \oplus w_i \oplus \cdots \oplus w_1 \big) \\
			&\quad = Z_k^* \big( \Delta_k \Theta_{k-1} \cdots \Theta_{j_i+1} f \oplus \cdots \oplus \Delta_{j_i+1} f \oplus g_{j_i} \oplus \cdots \oplus g_1 \big),
		\end{align*}
		where $g_l\in \ov{\im(\Delta_l)}$. By utilizing this equivalence, alongside the relation 
		\[
		\Theta_{J_r} \cdots \Theta_{J_{i+1}} = \Theta_k \cdots \Theta_{j_i+1},
		\]
		we conclude that $\mathcal{M}_{J_i} = \mathcal{M}_{j_i}$, which completes the proof.
	\end{proof}

	\noindent\textbf{Acknowledgment.}
The research of the first author is supported in part by the Indian Institute of Technology Goa (SEED Grant 2022/SG/KH/047) and the Anusandhan National Research Foundation (MATRICS Grant MTR/2022/000339). The second author is supported by a CSIR-SRF fellowship (File No. 09/1290(12920)/2021-EMR-I) from the Council of Scientific and Industrial Research (CSIR), India.
	% Bibliography - MODIFY THIS SECTION


\begin{thebibliography}{99}
	
	\bibitem{Bh05a}
	T. Bhattacharyya, J. Eschmeier and J. Sarkar, 
	\textit{Characteristic function of a pure commuting contractive tuple}, 
	Integral Equations Operator Theory, \textbf{53} (2005), 23--32.
	
	\bibitem{Bh06a}
	T. Bhattacharyya, J. Eschmeier and J. Sarkar, 
	\textit{On CNC commuting contractive tuples}, 
	Proc. Indian Acad. Sci. Math. Sci., \textbf{116} (2006), 299--316.
	
	\bibitem{Fr82a}
	A. E. Frazho, 
	\textit{Models for noncommuting operators}, 
	J. Funct. Anal., \textbf{48} (1982), 1--11.
	
	\bibitem{HM2026I}
	K. J. Haria and A. K. Maurya,
	\textit{$k$-regular factorizations and invariant subspaces of completely non-unitary contractions}, \url{https://doi.org/10.13140/RG.2.2.20073.45924}, preprint.
	
	\bibitem{Kerchy03}
	L. K\'erchy,
	\textit{On the factorization of operator-valued functions},
	Acta Scientiarum Mathematicarum (Szeged), \textbf{69}(1-2) (2003), 337--348.
	
	\bibitem{Po89a}
	G. Popescu, 
	\textit{Isometric dilations for infinite sequences of noncommuting operators}, 
	Transactions of the American Mathematical Society, \textbf{316} (1989), 523--536.
	
	\bibitem{Po89b}
	G. Popescu, 
	\textit{Characteristic functions for infinite sequences of noncommuting operators}, 
	J. Operator Theory, \textbf{22} (1989), 51--71.
	
	\bibitem{Po95a}
	G. Popescu, 
	\textit{Multi-analytic operators on Fock spaces}, 
	Math. Ann., \textbf{303} (1995), 31--46.
	
	\bibitem{Po06}
	G. Popescu, 
	\textit{Characteristic functions and joint invariant subspaces}, 
	J. Funct. Anal., \textbf{237} (2006), 277--320.
	
	\bibitem{Sz64a}
	B. Sz.-Nagy and C. Foia\c{s}, 
	\textit{Une caractérisation de sous-espaces invariants pour une contraction de l’espace de Hilbert}, 
	C. R. Math. Acad. Sci. Paris, \textbf{258} (1964), 3426--3429.
	
	\bibitem{Sz64b}
	B. Sz.-Nagy and C. Foia\c{s}, 
	\textit{Sur les contractions de l’espace de Hilbert. IX. Factorisations de la fonction caractéristique. Sous-espaces invariants}, 
	Acta Sci. Math. (Szeged), \textbf{25} (1964), 283--316.
	
	\bibitem{NF70}
	B. Sz.-Nagy and C. Foia\c{s},
	\textit{Harmonic Analysis of Operators on Hilbert Space},
	North-Holland--Akad\'emiai Kiad\'o, Amsterdam--Budapest, 1970.
	
	\bibitem{NF74}
	B. Sz.-Nagy and C. Foia\c{s},
	\textit{Regular factorizations of contractions},
	Proceedings of the American Mathematical Society, \textbf{43}(1) (1974), 91--93.
	
	\bibitem{NFBK10}
	B. Sz.-Nagy, C. Foia\c{s}, H. Bercovici, and L. K{\'e}rchy, 
	\textit{Harmonic analysis of operators on {H}ilbert space}, 
	2nd ed., Universitext, Springer, New York, 2010.
	

	
\end{thebibliography}
\end{document}